%% file: circuit_diameter.tex
\documentclass[11pt, a4paper]{article}

\usepackage[dvipsnames]{xcolor}
\usepackage{hyperref}

\usepackage[margin=2.5cm]{geometry}
\usepackage{mathtools}
\usepackage{amssymb}
\usepackage{amsthm}
\usepackage{amsmath}
\usepackage{mathrsfs}
\usepackage{empheq}

\usepackage{bbold}

\usepackage{array,multirow,blkarray}

\usepackage[toc,page]{appendix}

\usepackage[nodayofweek]{datetime}

\usepackage{textcomp}

\usepackage{thmtools} 
\usepackage{thm-restate}

\usepackage{bbm}
\usepackage[]{dsfont}
\usepackage{authblk}
\usepackage{setspace}
\usepackage[shortlabels]{enumitem}
\usepackage[
   disable,
colorinlistoftodos, shadow, linecolor=white, backgroundcolor=white]{todonotes}

\usepackage{upgreek}

\newcommand{\circuits}{\mathcal{C}}
\newcommand{\numit}{p}

\usepackage[linesnumbered,commentsnumbered,ruled,vlined]{algorithm2e}

\SetKw{run}{run}

\usepackage{bm}

\newcommand{\aug}{\operatorname{aug}}

\newcommand{\pr}[2]{\left\langle #1, #2 \right\rangle}
\newcommand{\supp}{\mathrm{supp}}
\newcommand{\1}{\mathbb{1}}
\newcommand{\0}{\mathbb{0}}
\newcommand{\rk}{\operatorname{rk}}

\newcommand{\set}[1]{\left\{ #1 \right\}}

\newcommand{\OPT}{\mathrm{OPT}}

\newcommand{\EE}{\mathcal E}

\newcommand{\poly}{\operatorname{poly}}

\usepackage{accents}

\newcounter{oraclecf}
\newcounter{algDe orithm saved}

 

\makeatletter
\let\original@algocf@latexcaption\algocf@latexcaption
\long\def\algocf@latexcaption#1[#2]{%
  \@ifundefined{NR@gettitle}{%
    \def\@currentlabelname{#2}%
  }{%
    \NR@gettitle{#2}%
  }%
  \original@algocf@latexcaption{#1}[{#2}]%
}
\makeatother

\usepackage{cleveref}

\usepackage[noadjust]{cite}
\usepackage{booktabs}

\usepackage{algorithmic}

\newcommand{\LP}{\operatorname{LP}}

\newcommand{\im}{\operatorname{Im}}

\newcommand\myshade{100}
\colorlet{mylinkcolor}{NavyBlue}
\colorlet{mycitecolor}{YellowOrange}
\colorlet{myurlcolor}{Aquamarine}

\hypersetup{
  linkcolor  = mylinkcolor!\myshade!black,
  citecolor  = mycitecolor!\myshade!black,
  urlcolor   = myurlcolor!\myshade!black,
  colorlinks = true,
}

\newtheorem{theorem}{Theorem}[section]
\newtheorem{lemma}[theorem]{Lemma}

\newtheorem{corollary}[theorem]{Corollary}

\newtheorem{claim}[theorem]{Claim}

\theoremstyle{definition}
\newtheorem{definition}[theorem]{Definition}

\newtheorem{remark}[theorem]{Remark}

\newtheoremstyle{case}{}{}{\itshape}{1em}{}{:}{ }{}
\theoremstyle{case}

\definecolor{darkred}{RGB}{180, 0, 0}
\definecolor{darkgreen}{rgb}{0, 0.6, 0}
\definecolor{darkblue}{RGB}{51,51,178}
\definecolor{lightgray}{RGB}{231,231,231}
\definecolor{lightblue}{RGB}{180,180,254}
\definecolor{lightred}{HTML}{FEB4B4}
\definecolor{darkcyan}{HTML}{7FBFBF}

\newcommand{\cl}{\operatorname{cl}}

\DeclareMathOperator*{\argmax}{arg\,max}
\DeclareMathOperator*{\argmin}{arg\,min}

\newcounter{myquestion}

\newcounter{mycomment}
\newcommand{\Comment}[2][BN]{%
\refstepcounter{mycomment}%
{%
\setstretch{0.7}%
\todo[inline, 
backgroundcolor=black!30!white
, size=\small]{%
\textbf{Comment[#1\themycomment]:} \\ #2}%
}}

\newcommand{\bnote}[2][]{
    \todo[#1,backgroundcolor=cyan!20!white,bordercolor=black,linecolor=cyan,size=\tiny]{BN: #2}}

\renewcommand{\phi}{\varphi}
\renewcommand{\rho}{\varrho}

\newcommand{\R}{\mathbb{R}}
\newcommand{\Z}{\mathbb{Z}}

\newcommand{\size}{\operatorname{size}}

\newcommand{\cL}{\mathcal L}

\usepackage[many]{tcolorbox}

\newtcolorbox[
  auto counter,
  number within=section,
  crefname={algorithm}{Algorithm}]%
{boxProblem}[2][]{
  colback=white,
  code={\def\mytitle{#2}},
  title={\bf Algorithm \thetcbcounter.} \mytitle,%
  sharp corners,
  #1}

  \newtcolorbox[
    auto counter,
    number within=section,
    crefname={subroutine}{Subroutine}]%
  {subroutineBox}[2][]{
    colback=white,
    code={\def\mytitle{#2}},
    title={\bf Subroutine \thetcbcounter.} \mytitle,%
    sharp corners,
    #1}

\newcounter{Hequation}

\makeatletter
\g@addto@macro\equation{\stepcounter{Hequation}}
\makeatother

\title{On Circuit Diameter Bounds via Circuit Imbalances\thanks{This project has received funding from the European Research Council (ERC) under the European Union's Horizon 2020 research and innovation programme (grant agreements ScaleOpt--757481 and QIP--805241). This work was done while the authors participated in the Discrete Optimization Trimester Program at the Hausdorff Institute for Mathematics in Bonn in 2021. An extended abstract of this paper has appeared in Proceedings of the 23rd Conference on Integer Programming and Combinatorial Optimization, IPCO 2022.}
}
\date{{\tt \{dadush,zhuan\}@cwi.nl, bnatura3@gatech.edu, l.vegh@lse.ac.uk}}

\author[1]{Daniel Dadush}
\author[1]{Zhuan Khye Koh}
\author[2]{Bento Natura} 
\author[3]{L{\'{a}}szl{\'{o}} A. V{\'{e}}gh}
\affil[1]{Centrum Wiskunde \& Informatica, The Netherlands}
\affil[2]{Georgia Institute of Technology, USA}
\affil[3]{London School of Economics and Political Science, UK} 

\begin{document}

\maketitle

\begin{abstract}
We study the circuit diameter of polyhedra, introduced by Borgwardt, Finhold, and Hemmecke (SIDMA 2015) as a relaxation of the combinatorial diameter. We show that the circuit diameter of a system $\{x\in \R^n:\, Ax=b,\,  \0\le x\le u\}$ for $A\in\R^{m\times n}$ is bounded by $O(m\min\{m, n - m\}\log(m+\kappa_A)+n\log n)$, where $\kappa_A$ is the circuit imbalance measure of the constraint matrix. This yields a strongly polynomial circuit diameter bound if e.g., all entries of $A$ have polynomially bounded encoding length in $n$. 
 Further, we present circuit augmentation algorithms for LPs using the minimum-ratio circuit cancelling rule. Even though the standard minimum-ratio circuit cancelling algorithm is  not  finite in general, our variant can solve an LP in $O(mn^2\log(n+\kappa_A))$ augmentation steps.
\end{abstract}

\input{intro}

\input{prelim}

\input{diam-bound}

\input{capacitated}

\input{feasibility}

\input{optimization}

\input{other-circuit}

\bibliographystyle{myalpha}
\bibliography{circuit_diameter}
\end{document}

%% file: intro.tex
\section{Introduction}
The \emph{combinatorial diameter} of a polyhedron $P$ is the diameter of the vertex-edge graph associated with $P$. Hirsch's famous conjecture from 1957 asserted that the combinatorial diameter of a $d$-dimensional polytope (bounded polyhedron) with $f$ facets is at most $f-d$. This was disproved by Santos in 2012 \cite{santos2012}. The \emph{polynomial Hirsch conjecture}, i.e., finding a poly$(f)$ bound on the combinatorial diameter remains a central question in the theory of linear programming.

The first quasipolynomial bound was given by Kalai and Kleitman \cite{Kalai1992,KK1992}, see \cite{Sukegawa2017} for the best current bound and an overview of the literature. Dyer and Frieze \cite{Dyer1994} proved the polynomial Hirsch conjecture for totally unimodular (TU) matrices. For a system $\{x\in\R^d:\, Mx\le b\}$ with integer constraint matrix $M$, polynomial diameter bounds were given in terms of the maximum subdeterminant $\Delta_M$ \cite{Bonifas2014,brunsch2013,eisenbrand2017,dadush2016shadow}. These arguments can be strengthened to using a parametrization by a `discrete curvature measure' $\delta_M\ge 1/(d\Delta^2_M)$. The best such bound was given by Dadush and H\"ahnle \cite{dadush2016shadow} as $O(d^3\log(d/\delta_M)/\delta_M)$, using a shadow vertex simplex algorithm.

\medskip

As a natural relaxation of the combinatorial diameter, Borgwardt, Finhold, and Hemmecke \cite{Borgwardt2015} initiated the study of \emph{circuit diameters}. 
Consider a polyhedron in standard equality form 
\begin{equation}\label{sys:polytope}\tag{P}
P=\{\,x\in \R^n: Ax=b, x\ge \0\,\}
\end{equation}
 for $A\in \R^{m\times n}$, $b\in \R^m$; we assume $\rk(A)=m$. For the linear space $W=\ker(A)\subseteq \R^n$, $g\in W$  is an 
 \emph{elementary vector} if $g$ is a support-minimal nonzero vector in $W$, that is, no $h\in W\setminus\{\0\}$ exists such that $\supp(h)\subsetneq \supp(g)$. A \emph{circuit} in $W$ is the support of some elementary vector; these are precisely the circuits of the associated  linear matroid $\mathcal{M}(A)$. 
 We remark that many papers on circuit diameter, e.g., \cite{Borgwardt2018circuit,Borgwardt2016-hierarchy,Borgwardt2015,Borgwardt2018,Kafer2019}, refer to elementary vectors as circuits; we follow the traditional convention of \cite{Fulkerson67,Rockafellar67,Lee89}.
 We let $\EE(W)=\EE(A)\subseteq W$ and $\circuits(W)=\circuits(A)\subseteq 2^n$
 denote the set of elementary vectors and circuits in the space $W=\ker(A)$, respectively.
All edge directions of $P$ are elementary vectors, and the set of elementary vectors $\EE(A)$ equals the set of all possible edge directions of $P$ in the form \eqref{sys:polytope} for varying $b\in\R^m$ \cite{Sturmfels1997}.

A \emph{circuit walk} is a sequence of points $x^{(0)},x^{(1)},\ldots,x^{(k)}$ in $P$ such that for each $i=0,\ldots,k-1$, $x^{(i+1)}=x^{(i)}+\alpha^{(i)} g^{(i)}$ for some $g^{(i)}\in \EE(A)$ and $\alpha^{(i)} > 0$, and further, 
$x^{(i)}+\alpha g^{(i)}\notin P$ for any $\alpha>\alpha^{(i)}$,
i.e., each consecutive circuit step is \emph{maximal}. The \emph{circuit diameter} of $P$ is the maximum length (number of steps) of a shortest circuit walk between any two vertices $x,y\in P$. Note that, in contrast to walks in the vertex-edge graph, circuit walks are non-reversible and the minimum length from $x$ to $y$ may be different from the one from $y$ to $x$; this is due to the maximal step requirement. The circuit-analogue of Hirsch conjecture, formulated in \cite{Borgwardt2015}, asserts that the circuit diameter of $d$-dimensional polyhedron with $f$ facets is at most $f-d$; this  may be true even for unbounded polyhedra, see \cite{Borgwardt2018circuit}.
 For $P$ in the form \eqref{sys:polytope}, $d=n-m$ and the number of facets is at most $n$; hence, the conjectured bound is $m$.

 Circuit diameter bounds have been shown for some combinatorial polytopes such as dual transportation polyhedra \cite{Borgwardt2015},  matching, travelling salesman, and fractional stable set polytopes  \cite{Kafer2019}.
The paper \cite{Borgwardt2016-hierarchy} introduced several other variants of circuit diameter, and explored the relation between them. We note that \cite{Borgwardt2016-hierarchy,Kafer2019,DKS19} considers circuits for LPs given in the general form
$\{x\in\R^n:\, Ax=b,\, Bx\le d\}$. In Section~\ref{sec:gen-circuits}, we show that this setting can be reduced to the form \eqref{sys:polytope}.

\paragraph{Circuit augmentation algorithms} Circuit diameter bounds are inherently related to \emph{circuit augmentation algorithms}. This is a general algorithmic scheme to solve an LP
\begin{equation}
  \label{sys:lp}\tag{LP}
  \min \; \pr{c}{x}  \quad \mathrm{s.t.}\quad
  Ax=b\, , \, 
  x \geq \0\, .\\
  \end{equation}
  The algorithm 
proceeds through a sequence of feasible solutions $x^{(t)}$. An initial feasible $x^{(0)}$ is required in the input. For $t=0,1,\ldots,$ the current $x^{(t)}$ is updated to $x^{(t+1)}=x^{(t)}+\alpha g$ for some $g\in \EE(A)$ such that $\pr{c}{g}\le 0$, and $\alpha>0$  such that $x^{(t)}+\alpha g$ is feasible. 
The elementary vector $g$ is an \emph{augmenting direction} if $\pr{c}{g}<0$ and such an $\alpha>0$ exists;
by  LP duality, $x^{(t)}$ is optimal if and only if no augmenting direction exists. 
The augmentation is \emph{maximal} if $x^{(t)}+\alpha' g$ is infeasible for any $\alpha'>\alpha$; $\alpha$ is called the maximal stepsize for $x^{(t)}$ and $g$. 
Clearly, an upper bound on the number of steps of a circuit augmentation algorithm with maximal augmentations for arbitrary cost $c$ and starting point  $x^{(0)}$ yields an upper bound on the circuit diameter.

Simplex is a circuit augmentation algorithm that is restricted to using special elementary vectors corresponding to edges of the polyhedron. Many network optimization algorithms can be seen as special circuit augmentation algorithms. 
Bland~\cite{Bland76} introduced a circuit augmentation algorithm for LP, that generalizes the  Edmonds--Karp--Dinic maximum flow algorithm and its analysis, %
see also \cite[Proposition 3.1]{Lee89}. Circuit augmentation algorithms were revisited by De Loera, Hemmecke, and Lee in 2015 \cite{DHL15}, analyzing different augmentation rules and also extending them to integer programming. 
De Loera, Kafer, and Sanit\`a \cite{DKS19} studied the convergence of these rules on 0/1-polytopes, as well as the computational complexity of performing them. 
We refer the reader to \cite{DHL15} and \cite{DKS19} for a more detailed overview of the background and history of circuit augmentations.

\paragraph{The circuit imbalance measure} For a linear space $W=\ker(A)\subseteq \R^n$, the \emph{circuit imbalance} $\kappa_W=\kappa_A$ is defined as the maximum of $|g_j/g_i|$ over all elementary vectors $g\in \EE(W)$, $i,j\in \supp(g)$. It can be shown that $\kappa_W=1$ if and only if $W$ is a unimodular space, i.e., the kernel of a totally unimodular matrix. This parameter and related variants have been used implicitly or explicitly in many areas of linear programming and discrete optimization, see \cite{Ekbatani2021} for a recent survey. It is closely related to the Dikin--Stewart--Todd condition number $\bar\chi_W$
 that  plays a key role in layered-least-squares interior point methods introduced by Vavasis and Ye \cite{Vavasis1996}. An LP of the form \eqref{sys:lp} for $A\in\R^{m\times n}$  can be solved in time poly$(n,m,\log\kappa_A)$, which is strongly polynomial if $\kappa_A \le 2^{\mathrm{poly}(n)}$; see \cite{DHNV20,DadushNV20} for recent developments and references.

\paragraph{Imbalance and diameter}  The combinatorial diameter bound $O(d^3\log(d/\delta_M)/\delta_M)$ from \cite{dadush2016shadow}  mentioned above translates to a bound  $O((n-m)^{3} m \kappa_{A}\log(\kappa_{A}+n))$ for the system in the form \eqref{sys:polytope}, see \cite{Ekbatani2021}. For circuit diameters, the Goldberg-Tarjan minimum-mean cycle cancelling algorithm for minimum-cost flows \cite{Goldberg89} naturally extends to a circuit augmentation algorithm for general LPs using the \emph{steepest-descent} rule. This yields a circuit diameter bound $O(n^2m\kappa_A \log(\kappa_A + n))$ \cite{Ekbatani2021}, see also \cite{Gauthier2021}. However, note that these bounds may be exponential in the bit-complexity of the input.

\subsection{Our contributions}
 Our first main contribution improves the $\kappa_A$ dependence to a $\log\kappa_A$ dependence for circuit diameter bounds.
\begin{theorem}\label{thm:circuit-diameter}
The circuit diameter of a system in the form \eqref{sys:polytope} with constraint matrix $A\in\R^{m\times n}$ is $O(m \min \{m, n- m\}\log(m+\kappa_A))$.
\end{theorem}
The proof in Section~\ref{sec:diam-bound} is via a simple `shoot towards the optimum' scheme. 
We need the well-known concept of \emph{conformal circuit decompositions}. We say that $x,y\in\R^n$ are \emph{sign-compatible} if $x_iy_i\ge 0$ for all $i\in[n]$. We write $x\sqsubseteq y$ if they are sign-compatible and further $|x_i|\le|y_i|$ for all $i\in [n]$. It follows from Carath\'{e}odory's theorem and Minkowski--Weyl theorem that for any linear space $W\subseteq \R^n$ and $x\in W$, there exists a decomposition $x=\sum_{j=1}^k h^{(j)}$ such that $h^{(j)}\in \EE(W)$, $h^{(j)}\sqsubseteq x$ for all $j\in [k]$ and $k\le \dim(W)$. This is called a \emph{conformal circuit decomposition} of $x$ (see also Definition~\ref{def:conformal} and Lemma~\ref{lem:conformal} below).

Let $B\subseteq [n]$ be a feasible basis and $N=[n]\setminus B$, i.e., $x^*=(A_B^{-1}b,\0_N) \ge \0_n$ is a basic feasible solution. This is the unique optimal solution to \eqref{sys:lp} for the cost function $c=(\0_B,\1_N)$. 
Let $x^{(0)}\in P$ be an arbitrary vertex.
We may assume that $n\leq 2m$, by restricting to the union of the support of $x^*$ and $x^{(0)}$, and setting all other variables to 0.
For the current iterate $x^{(t)}$, let us consider a conformal circuit decomposition 
$x^*-x^{(t)}= \sum_{j=1}^k h^{(j)}$. Note that the existence of such a decomposition \emph{does not} yield a circuit diameter bound of $n$, due to the maximality requirement in the definition of circuit walks. For each $j\in [k]$, $x^{(t)}+h^{(j)}\in P$, but there might be a larger augmentation $x^{(t)}+\alpha h^{(j)}\in P$ for some $\alpha>1$.

Still, one can use this decomposition to construct a circuit walk. Let us pick the most improving circuit from the decomposition, i.e., the one maximizing $-\pr{c}{h^{(j)}}=\|h^{(j)}_N\|_1$, and obtain $x^{(t+1)}=x^{(t)}+\alpha^{(t)} h^{(j)}$ for the maximum stepsize $\alpha^{(t)}\ge 1$. The proof of Theorem~\ref{thm:circuit-diameter} is based on analyzing this procedure. The first key observation is that $\pr{c}{x^{(t)}}=\|x^{(t)}_N\|_1$ decreases geometrically. 
Then, we look at the set of indices $L_t=\{i\in [n]:\, x^*_i>n\kappa_A\|x^{(t)}_N\|_1\}$ and $R_t=\{i\in [n]:\, x^{(t)}_i\le (n - m)x^*_i\}$, and show that indices may never leave these sets once they enter. Moreover, a new index is added to either set every $O(m\log(m+\kappa_A))$ iterations.
In Section~\ref{sec:capacity}, we extend this bound to the setting with upper bounds on the variables. 
\begin{theorem}\label{thm:circuit-diameter-bounds}
The circuit diameter of a system in the form $Ax=b$, $\0\le x\le u$ with constraint matrix $A\in\R^{m\times n}$ is $O(m \min\{m, n - m\}\log(m+\kappa_A) + (n-m)\log n)$.
\end{theorem}

There is a straightforward reduction from the capacitated form to \eqref{sys:polytope} by adding $n$ slack variables; however, this would give an $O(n^2\log(n+\kappa_A))$ bound. For the stronger bound, we use a  preprocessing that involves cancelling circuits in the support of the current solution; this eliminates all but $O(m)$ of the capacity bounds in $O(n\log n)$ iterations, independently of $\kappa_A$.

For rational input,  $\log(\kappa_A)=O(\size(A))$ where $\size(A)$ denotes the total encoding length of $A$ \cite{DHNV20}. Hence, our result yields an $O(m\min\{m, n - m\} \size(A)+n\log n)$ diameter bound on $Ax=b$, $\0\le x\le u$. This can be compared with 
the bounds $O(n \size(A,b))$ using deepest descent augmentation steps in \cite{DHL15,DKS19}, where $\size(A,b)$ is the encoding length of $(A,b)$. (Such a bound holds for every augmentation rule that decreases the optimality gap geometrically, including the minimum-ratio circuit rule discussed below.) Note that our bound is independent of $b$. Furthermore, it is also applicable to systems given by irrational inputs, in which case arguments based on subdeterminants and bit-complexity cannot be used.

In light of these results, the next important step towards the polynomial Hirsch conjecture might be to show a poly$(n,\log\kappa_A)$ bound on the combinatorial diameter of \eqref{sys:polytope}. Note that---in contrast with the circuit diameter---not even a poly$(n,\size(A,b))$ bound is known. In this context, the best known general bound is $O((n-m)^{3} m \kappa_{A}\log(\kappa_{A}+n))$ implied by \cite{dadush2016shadow}.

\paragraph{Circuit augmentation algorithms} 
The diameter bounds in Theorems~\ref{thm:circuit-diameter} and \ref{thm:circuit-diameter-bounds} rely on knowing the optimal solution $x^*$; thus, they do not provide efficient LP algorithms. We next present circuit augmentation algorithms with poly$(n,m,\log\kappa_A)$ bounds on the number of iterations. Such algorithms require  subroutines for finding augmenting circuits. In many cases, such subroutines are LPs themselves. However, they may be of a simpler form, and might be easier to solve in practice. Borgwardt and Viss \cite{BorgwardtV2020} exhibit an implementation of a steepest-descent circuit augmentation algorithm with encouraging computational results.

We assume that a subroutine
 \textsc{Ratio-Circuit}$(A,c,w)$ is available; this implements the well-known minimum-ratio circuit rule. It takes as input a matrix $A\in\R^{m\times n}$, 
$c\in\R^n$, $w\in (\R_{++}\cup\{\infty\})^n$, and returns a basic optimal solution to the system  
\begin{equation}
  \label{sys:minratio-i}
  \min \; \pr{c}{z}\, \quad \mathrm{s.t.}\quad
  Az =\0\, , \, 
  \pr{w}{z^-}  \leq 1\, ,
  \end{equation} 
where $(z^-)_i := \max\{0,-z_i\}$ for $i\in [n]$.
Here, we use the convention $w_iz_i = 0$ if $w_i = \infty$ and $z_i = 0$.
This system can be equivalently written as an LP using auxiliary variables.
If bounded, a basic optimal solution is either $\0$ or an elementary vector $z\in \EE(A)$ that minimizes $\pr{c}{z}/\pr{w}{z^-}$.

Given $x\in P$, we use weights $w_i=1/x_i$ (with $w_i=\infty$ if $x_i=0$). For minimum-cost flow problems, this rule was proposed by Wallacher \cite{Wallacher}; such a cycle can be found in strongly polynomial time for flows. The main advantage of this rule is that the optimality gap decreases by a factor $1-1/n$ in every iteration. This rule, along with the same convergence property,
can be naturally extended to  linear programming \cite{mccormick-shioura-not-strongly}, and has found several combinatorial applications, e.g., \cite{Wallacher1999,Wayne02}, and has also been used in the context of integer programming \cite{Schulz1999}. 

On the negative side, Wallacher's algorithm is \emph{not} strongly polynomial: it does not terminate finitely for minimum-cost flows, as shown in \cite{mccormick-shioura-not-strongly}. In contrast, our algorithms 
achieve a strongly polynomial running time whenever $\kappa_A\le 2^{\mathrm{poly}(n)}$. An important modification is the occasional use of a second type of circuit augmentation step \textsc{Support-Circuit} that removes circuits in the support of the current (non-basic) iterate $x^{(t)}$ (see \Cref{subroutine:support-circuit}); this can be implemented using simple linear algebra. 
Our first result addresses the feasibility setting:
\begin{theorem}\label{thm:feasible-augment}
Consider an LP of the form \eqref{sys:lp} with cost function 
$c = (\0_{[n]\setminus N}, \1_N)$ for some $N\subseteq [n]$. There exists a circuit augmentation algorithm that either finds a solution $x$ such that $x_N=\0$ or a dual certificate that no such solution exists, using  $O(mn\log(n+\kappa_A))$ \textsc{Ratio-Circuit} and $(m+1)n$ \textsc{Support-Circuit} augmentation steps.
\end{theorem}
 Such problems typically arise in Phase I of the Simplex method when we add auxiliary variables in order to find a feasible solution. 
 The algorithm is presented in Section~\ref{sec:feas}.  The analysis extends that of Theorem~\ref{thm:circuit-diameter}, tracking large coordinates $x_i^{(t)}$.
Our second result considers general optimization:
\begin{theorem}\label{thm:opt-augment}
Consider an LP of the form \eqref{sys:lp}. There exists a circuit augmentation algorithm that finds an optimal solution or concludes unboundedness using $O(mn^2\log(n+\kappa_A))$  \textsc{Ratio-Circuit} and  $(m+1)n^2$ \textsc{Support-Circuit} augmentation steps. 
\end{theorem}
The proof is given in Section~\ref{sec:opt}. The main subroutine  identifies a new index $i\in[n]$ such that $x^{(t)}_i=0$ in the current iteration and $x^*_i=0$ in an optimal solution; we henceforth fix this variable to 0.  To derive this conclusion, at the end of each phase the current iterate $x^{(t)}$  will be optimal to \eqref{sys:lp} with a slightly modified cost function $\tilde c$; the conclusion follows using a proximity argument
(\Cref{thm:fixing}). The overall algorithm repeats this subroutine $n$ times.  The subroutine is reminiscent of the feasibility algorithm (\Cref{thm:feasible-augment}) with the following main difference:  whenever we identify a new `large' coordinate, we slightly perturb the cost function.

\paragraph{Comparison to black-box LP approaches}  An important milestone towards strongly polynomial linear programming was Tardos's 1986 paper \cite{Tardos86} on solving \eqref{sys:lp} in time poly$(n,m,\log\Delta_A)$, where $\Delta_A$ is the maximum subdeterminant of $A$. Her algorithm makes  $O(nm)$ calls to a weakly polynomial LP solver for instances with small integer capacities and
costs, and uses proximity arguments to gradually learn the support of an optimal solution. This approach was extended to the real model of computation for a poly$(n,m,\log\kappa_A)$ bound \cite{DadushNV20}. The latter result uses proximity arguments with circuit imbalances $\kappa_A$, and eliminates all dependence on bit-complexity. 

The proximity tool \Cref{thm:fixing} derives from \cite{DadushNV20}, and our circuit augmentation algorithms are inspired by the feasibility and optimization algorithms in this paper. However, using circuit augmentation oracles instead of an approximate LP oracle changes the setup. Our arguments become simpler since we proceed through a sequence of feasible solutions, whereas much effort in \cite{DadushNV20} is needed to deal with infeasibility of the solutions returned by the approximate solver. On the other hand, we need to be more careful as all steps must be implemented using circuit augmentations in the original system, in contrast to the higher degree of freedom in \cite{DadushNV20} where we can make approximate solver calls to arbitrary modified versions of the input LP.

\paragraph{Organization of the paper} The rest of the paper is organized as follows. 
We first provide the necessary preliminaries in \Cref{sec:prelim}.
In \Cref{sec:diam-bound}, we upper bound the circuit diameter of \eqref{sys:polytope}.
In \Cref{sec:capacity}, this bound is extended to the setting with upper bounds on the variables.
Then, we develop circuit-augmentation algorithms for solving \eqref{sys:lp}.
In particular, \Cref{sec:feas} contains the algorithm for finding a feasible solution, whereas \Cref{sec:opt} contains the algorithm for solving \eqref{sys:lp} given an initial feasible solution. \Cref{sec:gen-circuits} shows how circuits in LPs of more general forms can be reduced to the notion used in this paper.

%% file: prelim.tex
\section{Preliminaries} \label{sec:prelim}
Let $[n]=\{1,2,\ldots,n\}$.
Let $\R_+$ and $\R_{++}$ be the set of nonnegative and positive real numbers respectively.
For $\alpha\in\R$, we denote $\alpha^+=\max\{0,\alpha\}$ and $\alpha^-=\max\{0,-\alpha\}$. 
For a vector $z \in \R^n$, we define $z^+,z^-\in\R^n$ as $(z^+)_i=(z_i)^+$, $(z^-)_i=(z_i)^-$ for $i\in [n]$. 
For $z\in\R^n$, we let $\supp(z)=\{i\in [n]: z_i\neq 0\}$ denote its support, and $1/z\in(\R\cup\{\infty\})^n$ denote the vector $(1/z_i)_{i\in [n]}$.
 We use
$\|\cdot\|_p$ to denote the $\ell_p$-norm; we simply
write $\|\cdot\|$ for $\|\cdot\|_2$.  For $A\in\R^{m\times n}$ and $S\subseteq [n]$, we let $A_S\in \R^{m\times |S|}$ denote the submatrix corresponding to columns $S$. 
We denote $\rk(S):= \rk(A_S)$, i.e., the rank of the set $S$ in the linear matroid associated with $A$.
A \emph{spanning subset of $S$ is a subset $T\subseteq S$ such that $\rk(T) = \rk(S)$.}
The \emph{closure} of $S$ is defined as $\cl(S):= \set{i\in [n]:\rk(S\cup\{i\}) = \rk(S)}$.
The dual linear program of \eqref{sys:lp} is 
\begin{equation}\tag{DLP}
  \label{sys:dual_lp}
  \max \; \pr{b}{y}  \quad \mathrm{s.t.}\quad
  A^\top y + s=c\, , \, 
  s \geq \0\, .
\end{equation}
Note that $y$ uniquely determines $s$, and due to the assumption $\rk(A) =m$, $s$ also uniquely determines $y$. %
For this reason, given a dual feasible solution $(y,s)$, we can just focus on $y$ or $s$.

For $A\in \R^{m\times n}$, let $W=\ker(A)$.
Recall that $\mathcal{C}(W) = \mathcal{C}(A)$ and $\EE(W)= \EE(A)$ are the set of circuits and elementary vectors in $W$ respectively.
Note that every circuit has size at most $m+1$ because we assumed that $\rk(A) = m$.
The \emph{circuit imbalance measure} of $W$ is defined as 
\[\kappa_W \coloneqq \kappa_A \coloneqq \max_{g\in \EE(W)} \set{\frac{|g_i|}{|g_j|}:i,j\in \supp(g)}\]
if $W\neq \{\0\}$. Otherwise, it is defined as $\kappa_W\coloneqq \kappa_A\coloneqq 1$. For a linear space $W\subseteq \R^n$, let $W^\perp$ denote the orthogonal complement. Thus, for $W=\ker(A)$, $W^\perp=\im(A^\top)$.
According to the next lemma, circuit imbalances are self-dual.
\begin{lemma}[{\cite{DHNV20}}]\label{lem:duality}
For a linear space $W\subseteq\R^n$, we have $\kappa_W=\kappa_{W^\perp}$. %
\end{lemma}

For $P$ as in \eqref{sys:polytope}, $x \in P$ and an elementary vector $g \in \EE(A)\setminus \R^n_+$, we let $\aug_P(x, g) := x + \alpha g$ where $\alpha = \max\{\bar \alpha : x + \bar \alpha g \in P\}$. 

\begin{definition}\cite{bookDeLoera}\label{def:conformal}
  We say that $x,y\in\R^n$ are \emph{sign-compatible} if $x_iy_i\geq 0$ for all $i\in[n]$. %
  We write $x\sqsubseteq y$ if they are sign-compatible and further $|x_i|\le|y_i|$ for all $i\in [n]$.  
  For a linear space $W\subseteq \R^n$ and $x\in W$, a \emph{conformal circuit decomposition} of $x$ is a set of elementary vectors $h^{(1)},h^{(2)},\dots,h^{(k)}$ in $W$ such that $x=\sum_{j=1}^k h^{(j)}$, $k\le \dim(W)$, and $h^{(j)}\sqsubseteq x$ for all $j\in [k]$.
\end{definition}

The following lemma shows that every vector in a linear space has a conformal circuit decomposition. It is a simple corollary of the Minkowski--Weyl and Carath\'eodory theorems.
\begin{lemma}\label{lem:conformal}
For a linear space $W\subseteq \R^n$,  every $x\in W$ has a conformal circuit decomposition $x = \sum_{j=1}^k h^{(j)}$ such that $k \le \min\{\dim(W), |\supp(x)|\}$.
\end{lemma}
\subsection{Circuit oracles}
In Sections~\ref{sec:capacity}, \ref{sec:feas}, and \ref{sec:opt}, we use a simple circuit finding subroutine \textsc{Support-Circuit}$(A,c,x,S)$ that will be used to identify circuits in the support of a solution $x$.  This can be implemented easily using Gaussian elimination.
Note that the constraint $\pr{c}{z}\leq 0$ is superficial as $-z$ is also an elementary vector for every elementary vector $z$.

\begin{subroutineBox}[label={subroutine:support-circuit}]{\textsc{Support-Circuit}$(A,c,x,S)$}
  For a  matrix $A\in\R^{m\times n}$,  vectors $c,x\in \R^n$ and $S\subseteq [n]$, the output is an elementary vector $z\in\EE(A)$ with  $\supp(z)\subseteq \supp(x)$, $\supp(z)\cap S\neq \emptyset$ with $\pr{c}{z}\le 0$, or concludes that no such elementary vector exists. 
\end{subroutineBox}

\medskip

The circuit augmentation algorithms in Sections~\ref{sec:feas} and \ref{sec:opt} will use the 
 subroutine \textsc{Ratio-Circuit}$(A,c,w)$.

\begin{subroutineBox}[label={subroutine:ratio-circuit}]{\textsc{Ratio-Circuit}$(A,c,w)$}
  For a matrix $A\in\R^{m\times n}$ and vectors $c\in\R^n$, $w\in (\R_{++}\cup\{\infty\})^n$, the output is a basic optimal solution
  to the system:
  \begin{equation}
    \label{sys:minratio}
    \min \; \pr{c}{z}\, \quad \mathrm{s.t.}\quad
    Az =\0\, , \, 
    \pr{w}{z^-}  \leq 1\, ,
    \end{equation} 
and an optimal solution to the following dual program:
  \begin{equation}
  \label{sys:minratio-dual}
  \max \; -\lambda \quad \mathrm{s.t.}\quad A^\top y+ s=c\, \quad
 \0\le s\le \lambda w
  \end{equation}
\end{subroutineBox}

Note that \eqref{sys:minratio} can be reformulated as an LP using additional variables, and its dual LP can be equivalently written as \eqref{sys:minratio-dual}. 
Recall that we use the convention $w_iz_i = 0$ if $w_i = \infty$ and $z_i = 0$ in \eqref{sys:minratio}.
The opposite convention is used in \eqref{sys:minratio-dual}, i.e., $\lambda_i w_i = \infty$ if $\lambda = 0$ and $w_i = \infty$.
If \eqref{sys:minratio} is bounded, then a basic optimal solution is either $\0$ or an elementary vector $z\in\EE(A)$ that minimizes $\pr{c}{z}/\pr{w}{z^-}$. 
Moreover, observe that every feasible solution to $\eqref{sys:minratio-dual}$ is also feasible to \eqref{sys:dual_lp}.

We will use the following lemma, a direct consequence of \cite[Lemma~4.3]{DNV20}.
\begin{lemma}\label{cor:construct_hoffman}
Given $A\in\R^{m\times n}$, $W=\ker(A)$, $\ell\in (\R\cup\{-\infty\})^n$ and $u\in (\R\cup \{\infty\})^n$, let $r\in W$ such that $\ell \leq r\leq u$. In $\poly(m,n)$ time, we can find a vector $r'\in W$ such that $\ell \leq r'\leq u$ and $\|r'\|_\infty \leq \kappa_A \|\ell^+ + u^-\|_1$.
\end{lemma}
This lemma, together with \Cref{lem:duality}, allows us to 
 assume that the optimal dual solution $s$ returned by {\sc Ratio-Circuit} satisfies 
\begin{equation}\label{eq:minratio_dual_prox}
  \|s\|_\infty \leq 2\kappa_A \|c\|_1.
\end{equation}

To see this, let $(y,s,\lambda)$ be an optimal solution to \eqref{sys:minratio-dual}.
We know that $-c \leq s-c \leq \lambda w - c$.
Let $\ell \coloneqq -c$, $r \coloneqq s-c$ and $u \coloneqq \lambda w - c$.
By \Cref{cor:construct_hoffman}, we can in $\poly(m,n)$ time compute $r'\in W^\perp$ such that $\ell \leq r'\leq u$ and
\[\|r'\|_\infty \leq \kappa_{W^\perp} \|\ell^+ + u^-\|_1 \leq \kappa_{W^\perp} \|c^- + c^+\|_1  = \kappa_{W^\perp} \|c\|_1 .\]
Then, $s'\coloneqq r'+c$ is an optimal solution to \eqref{sys:minratio-dual} which satisfies 
\[\|s'\|_\infty \leq \|r'\|_\infty + \|c\|_\infty \leq (\kappa_{W^\perp}+1)\|c\|_1 \leq 2\kappa_{W^\perp}\|c\|_1.\]
Thus, \eqref{eq:minratio_dual_prox} follows using \Cref{lem:duality}, since $\kappa_{W^\perp}=\kappa_W=\kappa_A$.

The following lemma is well-known, see e.g., \cite[Lemma 2.2]{mccormick-shioura-not-strongly}.

\begin{lemma}\label{lem:progress}
Let $\OPT$ be the optimal value of \eqref{sys:lp}, and assume that it is finite.
Given a feasible solution $x$ to \eqref{sys:lp}, let $g$ be the optimal solution to \eqref{sys:minratio} returned by {\sc Ratio-Circuit}$(A,c,1/x)$. 
\begin{enumerate}[label=(\roman*)]
  \item If $\pr{c}{g} = 0$, then $x$ is optimal to \eqref{sys:lp}.
  \item If $\pr{c}{g}<0$, then letting $x'=\aug_P(x, g)$, we have $\alpha\geq 1$ for the augmentation stepsize and
  \[\pr{c}{x'}-\OPT\le \left(1-\frac{1}{|\supp(x)|}\right)\left(\pr{c}{x}-\OPT\right)\, .\]
\end{enumerate}
\end{lemma}
\begin{proof}
We only prove (ii) because (i) is trivial.
The stepsize bound $\alpha\ge 1$ follows since $\pr{1/x}{g^-}\le 1$; thus, $x+g\in P$.
Let $x^*$ be an optimal solution to \eqref{sys:lp}, and let $z=(x^*-x)/|\supp(x)|$. 
Note that $g\ngeq \0$, as otherwise \eqref{sys:minratio} is unbounded. So, $x\neq \0$.
Then, $z$ is feasible to 
\eqref{sys:minratio} for $w=1/x$. Therefore,
\[
\alpha\pr{c}{g}\le \pr{c}{g}\le \pr{c}{z}=\frac{\OPT-\pr{c}{x}}{|\supp(x)|}\, ,
\]
implying the second claim.
\end{proof}

\begin{remark}
It is worth noting that \Cref{lem:progress} shows that applying \textsc{Ratio-Circuit} to vectors $x$ with small support gives better convergence guarantees. \Cref{alg:feas,alg:circ_aug_opt} for feasibility and optimization in \Cref{sec:feas,sec:opt} apply \textsc{Ratio-Circuit} to vectors $x$ which have large support $|\supp(x)| = \Theta(n)$ in general. These algorithms could be reformulated in that one first runs \textsc{Support-Circuit} to reduce the size of the support to size $O(m)$ and only then runs \textsc{Ratio-Circuit}. The guarantees of \Cref{lem:progress} now imply that to reduce the optimality gap by a constant factor we would replace $O(n)$ calls to \textsc{Ratio-Circuit} with only $O(m)$ calls. On the other hand, this comes at the cost of $n$ additional  calls to \textsc{Support-Circuit} for every call to \textsc{Ratio-Circuit}. 
\end{remark}

\subsection{A norm bound}
We now formulate a proximity bound asserting that if the columns of $A$ outside $N$ are linearly independent, then we can bound the $\ell_\infty$-norm of any vector in $\ker(A)$ by the norm of its coordinates in $N$. This can be seen as a special case of Hoffman-proximity results; see Section~\ref{sec:proximity} for more such results and references.
\begin{lemma}\label{lem:circuit-N-kappa}
For $A\in\R^{m\times n}$, let $N\subseteq[n]$ such that $A_{[n]\setminus N}$ has full column rank. Then, for any $z\in \ker(A)$, we have $\|z\|_\infty\le \kappa_A\|z_N\|_1$.
\end{lemma}
\begin{proof}
Let $h^{(1)},\ldots, h^{(k)}$ be a conformal circuit decomposition of $z$. 
Then, $\|z\|_\infty\leq \sum_{t=1}^k \|h^{(t)}\|_\infty$.
For each $h^{(t)}$, we have $\supp(h^{(t)})\cap N\neq\emptyset$ because $A_{[n]\setminus N}$ has full column rank.
Hence, $\|h^{(t)}\|_\infty\le \kappa_A |h^{(t)}_{j(t)}|$ for some $j(t)\in N$. 
Conformality implies that
\[\sum_{t=1}^k|h^{(t)}_{j(t)}| = \sum_{s\in N}\sum_{j(t) = s}|h^{(t)}_{j(t)}| \leq \sum_{s\in N}|z_s| = \|z_N\|_1.\]
The lemma follows by combining all the previous inequalities.
\end{proof}

\subsection{Estimating circuit imbalances}
The circuit augmentation algorithms in Sections~\ref{sec:feas} and \ref{sec:opt} explicitly use the circuit imbalance measure $\kappa_A$. However, this is NP-hard to approximate within a factor $2^{O(n)}$, see \cite{Tuncel1999,DHNV20}. We circumvent this problem using a standard guessing procedure, see e.g., \cite{Vavasis1996,DHNV20}. Instead of $\kappa_A$, we use an estimate $\hat\kappa$, initialized as $\hat\kappa=n$. Running the algorithm with this estimate either finds the desired feasible or optimal solution (which one can verify), or fails. In case of failure, we conclude that $\hat\kappa<\kappa_A$, and replace $\hat\kappa$ by $\hat\kappa^2$. 
Since the running time of the algorithms is linear in $\log(n+\hat\kappa)$, the running time of all runs will be dominated by the last run, giving the desired bound.
For simplicity, the algorithm descriptions use the explicit value $\kappa_A$.

%% file: diam-bound.tex
\section{The Circuit Diameter Bound}\label{sec:diam-bound}

In this section, we show Theorem~\ref{thm:circuit-diameter}, namely the bound $O(m\min\{m,n-m\}\log(m+\kappa_A))$ on the circuit diameter of a polyhedron in standard form \eqref{sys:polytope}.
As outlined in the Introduction, let $B\subseteq [n]$ be a feasible basis and $N=[n]\setminus B$ such that $x^*=(A_B^{-1}b,\0_N)$ is a basic solution to \eqref{sys:lp}. 
We can assume $n\le 2m$: the union of the supports of the starting vertex $x^{(0)}$ and the target vertex $x^*$ is at most $2m$; we can fix all other variables to 0. 
Defining $\tilde n := |\supp(x^*) \cup \supp(x^{(0)})| \le 2m$ and restricting $A$ to these columns, we show a circuit diameter bound
$O(\tilde n(\tilde n-m)\log(m+\kappa_A))$. This implies
\Cref{thm:circuit-diameter} for general $n$.
In the rest of this section, we use $n$ instead of $\tilde n$, but assume $n\le 2m$.
The simple `shoot towards the optimum' procedure is shown in
 \Cref{proc:diam-bound}. 

\begin{algorithm}[htb!]
  \caption{\textsc{Diameter-Bound}}
  \label{proc:diam-bound}
  \SetKwInOut{Input}{Input}
  \SetKwInOut{Output}{Output}
  \SetKwComment{Comment}{$\triangleright$\ }{}
  \SetKw{And}{\textbf{and}}
  \SetKw{Or}{\textbf{or}}
  \Input{Polyhedron in standard form \eqref{sys:polytope}, basis $B\subseteq [n]$ with its corresponding vertex $x^* = (A^{-1}_Bb, \0_N)$, and initial vertex $x^{(0)}$.}
  \Output{Length of a circuit walk from $x^{(0)}$ to $x^*$.}
  $t\gets 0$\;
  \While{$x^{(t)}\neq x^*$}{
    Let $h^{(1)},h^{(2)},\dots,h^{(k)}$ be a conformal circuit decomposition of $x^* - x^{(t)}$\;
    $g^{(t)}\gets h^{(j)}$ for any $j\in \argmax_{i\in [k]}\|h^{(i)}_N\|_1$\;
    $x^{(t+1)} \gets \aug_P(x^{(t)}, g^{(t)})$; $t\gets t+1$  \; 
    }
  \Return $t$ \;
\end{algorithm}

A priori, even finite termination is not clear. 
First, we show that the `cost' $\|x^{(t)}_N\|_1$ decreases geometrically.
It is a consequence of choosing the most improving circuit $g^{(t)}$ in each iteration.

\begin{lemma}\label{lem:geo-decay}
For every iteration $t\ge 0$, we have $\|x^{(t+1)}_N\|_1 \leq (1-\frac{1}{n-m})\|x_N^{(t)}\|_1$.
Furthermore, $|x_i^{(t+1)} - x_i^{(t)}| \le (n - m) |x_i^* - x_i^{(t)}|$ for all $i \in [n]$.
\end{lemma}

\begin{proof}
Let $h^{(1)}, \ldots, h^{(k)}$ with $k \le n - m$ be the conformal circuit decomposition of $x^* - x^{(t)}$ used in iteration $t$ of \Cref{proc:diam-bound}. 
Note that $h^{(i)}_N \le \0_N$ for all $i \in [k]$ because $x_N^* = \0_N$ and $x^{(t)} \ge \0$.
By our choice of $g^{(t)}$,
\[\|g_N^{(t)}\|_1 = \max_{i \in [k]} \|h^{(i)}_N\|_1 \ge \frac1k \sum_{i \in [k]} \|h^{(i)}_N\|_1 = \frac1k \|x_N^{(t)}\|_1\] 
where the last equality uses the conformality of the decomposition.
Let $\alpha^{(t)}$ be such that $x^{(t+1)} = x^{(t)} + \alpha^{(t)}g^{(t)}$. Clearly, $\alpha^{(t)}\ge 1$ because $x^{(t)}+g^{(t)}\in P$.
Hence,
\begin{equation*}
\begin{aligned}
\label{eq:circ_progress}
\|x^{(t+1)}_N\|_1& = \|x^{(t)}_N + \alpha^{(t)}g^{(t)}_N\|_1
\le \|x_N^{(t)} + g_N^{(t)}\|_1 \\
&=\|x_N^{(t)}\|_1- \| g_N^{(t)}\|_1\le \big(1-\frac1k\big) \|x_N^{(t)}\|_1\, .
\end{aligned}
\end{equation*}

Further, using $\0\le x_N^{(t+1)}\le x_N^{(t)}$,  we see that
\begin{equation*}
  \alpha^{(t)} = \frac{\|x_N^{(t+1)} - x_N^{(t)}\|_1}{\|g_N^{(t)}\|_1} \le  \frac{\|x_N^{(t)}\|_1}{\|g_N^{(t)}\|_1} \le k, 
\end{equation*} 
and so for all $i\in [n]$ we have 
$|x_i^{(t+1)} - x_i^{(t)}| = \alpha^{(t)} |g_i^{(t)}| \le k |g_i^{(t)}| \le k |x_i^* - x_i^{(t)}|
$.
\end{proof}

Our convergence proof is based on analyzing the following sets
\begin{equation*}
L_t:=\{i\in [n]:\, x_i^* > n\kappa_A \|x_N^{(t)}\|_1\}\, , \qquad T_t:=[n]\setminus L_t\, ,\qquad R_t:= \{i\in [n]:x^{(t)}_i \leq (n-m)x^*_i\}\, .
\end{equation*}
The set $L_t$ consists of indices $i$ where $x_i^*$ is much larger than the current `cost' $\|x_N^{(t)}\|_1$.
On the other hand, the set $R_t$ consists of indices $i$ where $x_i^{(t)}$ is not much above $x_i^*$. 
The next lemma shows that the sets $L_t$ and $R_t$ are monotonically growing.

\begin{lemma}\label{lem:I-monotone}
For every iteration $t\ge 0$, we have $L_t\subseteq L_{t+1}\subseteq B$ and $R_t\subseteq R_{t+1}$.
\end{lemma}

\begin{proof}
Clearly, $L_t \subseteq L_{t+1}$ as $\|x_N^{(t)}\|_1$ is monotonically decreasing by \Cref{lem:geo-decay}, and $L_t \subseteq B$ as $x_N^* = \0_N$.
Next, let $j \in R_t$.
If $x_j^{(t)} \ge x_j^*$, then $x_j^{(t+1)} \le x_j^{(t)}$ by conformality. 
If $x_j^{(t)} < x_j^*$, then $x^{(t+1)}_j \le x^{(t)}_j + (n -m) (x_j^* - x_j^{(t)}) \le (n - m)x_j^*$ by \Cref{lem:geo-decay}. 
In both cases, we conclude that $j \in R_{t+1}$.  
\end{proof}

Our goal is to show that $R_t$ or $L_t$ is extended within $O((n-m)\log(n+\kappa_A))$ iterations.
By the maximality of the augmentation, we know that at least one variable is set to zero in every iteration $t$. The following lemma shows that these variables do not lie in $L_t$.

\begin{lemma}\label{lem:zeroed-vars}
For every iteration $t\geq 0$, we have $\emptyset\neq \supp(x^{(t)})\setminus \supp(x^{(t+1)}) \subseteq T_t$.
\end{lemma}

\begin{proof}
Let $i \in \supp(x^{(t)}) \setminus \supp(x^{(t+1)})$. 
Such a variable exists by the maximality of the augmentation. 
Clearly, $x^{(t+1)}_i = 0$.
Applying Lemma~\ref{lem:circuit-N-kappa} to $x^{(t+1)} - x^*\in\ker(A)$ yields
\begin{equation*}
x^*_i \leq \|x^{(t+1)} - x^*\|_\infty \le \kappa_A \|x_N^{(t+1)} - x_N^*\|_1 = \kappa_A\|x_N^{(t+1)}\|_1 \leq \kappa_A\|x_N^{(t)}\|_1.
\end{equation*}
The equality is due to $x^*_N = \0$, while the last inequality follows from \Cref{lem:geo-decay}.
So, $i \in T_t$. 
\end{proof}

Clearly, any variable $i$ that is set to zero in iteration $t$ belongs to $R_{t+1}$.
So, if $i\notin R_t$, then we make progress as $R_t\subsetneq R_{t+1}$.
Note that this is always the case if $i\in N$.
We show that if $\|x^{(t)}_{T_t} - x^*_{T_t}\|_\infty$ is sufficiently large, then $i\notin R_t$.

\begin{lemma}\label{lem:C-grow}
If $\|x^{(t)}_{T_t} - x^*_{T_t}\|_\infty > 2mn^2\kappa_A^2\left\|x^*_{T_t}\right\|_\infty$ for some iteration $t$, then $R_t \subsetneq R_{t+1}$.
\end{lemma}

\begin{proof}
Let $i\in \supp(x^{(t)})\setminus \supp(x^{(t+1)})$.
Clearly, $i\in R_{t+1}$ because $x^{(t+1)}_i = 0$.
So, it suffices to show that $i\notin R_t$.
Since $x^{(t+1)} - x^{(t)}$ is an elementary vector, we have $\|x^{(t+1)} - x^{(t)}\|_\infty \leq \kappa_A |x^{(t+1)}_i - x^{(t)}_i| = \kappa_A x^{(t)}_i$. 
As $|\supp(x^{(t+1)}-x^{(t)})|\leq m+1$, we obtain
\begin{equation}
  \label{eq:bound1}
\|x_N^{(t)} - x_N^{(t+1)}\|_1 \le (m+1) \|x_i^{(t)} - x_i^{(t+1)}\|_\infty \leq
 (m+1)\kappa_A x_i^{(t)}\le 2m\kappa_A  x_i^{(t)}\, .
\end{equation}

Let $h^{(1)}, \ldots, h^{(k)}$ with $k \le n -m $ be the conformal circuit decomposition of $x^* - x^{(t)}$ used in iteration $t$ of \Cref{proc:diam-bound}. Let $j \in T_t$ such that $|x^{(t)}_j - x_j^*| = \|x_{T_t}^{(t)} - x_{T_t}^*\|_\infty$. There exists $\widetilde h=h^{(\ell)}$ for some $\ell\in [k]$ in this decomposition such that $|\widetilde h_j| \ge \frac1k|x^{(t)}_j - x^*_j|$. Since $A_B$ has full column rank, we have $\supp(\widetilde h) \cap N \neq \emptyset$ and so
\begin{equation}
  \label{eq:bound2}
  \|\widetilde h_N\|_1 \ge \frac{|\widetilde h_j|}{\kappa_A} \ge \frac{|x^{(t)}_j - x^*_j|}{k\kappa_A}\, .
\end{equation}
From \eqref{eq:bound1}, \eqref{eq:bound2} and noting that $\|\widetilde h_N\|_1 \le \|g^{(t)}_N\|_1\le \|x_N^{(t)} - x_N^{(t+1)}\|_1$  by our choice of $g^{(t)}$, we get
\begin{equation*}
 x_i^{(t)} \ge \frac{\|x_N^{(t)} - x_N^{(t+1)}\|_1}{2m\kappa_A} \ge \frac{\|\widetilde h_N\|_1}{2m\kappa_A} \ge \frac{\|x_{T_t}^{(t)} - x_{T_t}^*\|_\infty}{2mk\kappa_A^2}\, .
\end{equation*}

Thus, if $\|x_{T_t}^{(t)} - x_{T_t}^*\|_\infty > 2mn^2\kappa_A^2 \|x_{T_t}^*\|_\infty$ as in the assumption of the lemma, then $x_i^{(t)} > n\|x_{T_t}^*\|_\infty \ge n x_i^*$, where the last inequality is due to $i\in T_t$ by \Cref{lem:zeroed-vars}. 
It follows that $i \notin R_t$ as desired.
\end{proof}

We are ready to give the convergence bound. 
We have just proved that a large $\|x^{(t)}_{T_t} - x^*_{T_t}\|_\infty$ guarantees the extension of $R_t$. Using the geometric decay of $\|x_N^{(t)}\|$ (\Cref{lem:geo-decay}), we now show that if $\|x^{(t)}_{T_t} - x^*_{T_t}\|_\infty$ is small, then $\|x^{(t)}_N\|_1$ drops sufficiently such that a new variable enters $L_t$.

\begin{proof}[Proof of Theorem~\ref{thm:circuit-diameter}]
Recall that we assumed $n\le 2m$ without loss of generality.
In light of \Cref{lem:I-monotone}, it suffices to show that either $L_t$ or $R_t$ is extended in every $O((n-m)\log (n+\kappa_A))$ iterations.
If $\|x^{(t)}_{T_t} - x^*_{T_t}\|_\infty > 2mn^2\kappa_A^2 \left\|x^*_{T_t}\right\|_\infty$, then $R_t \subsetneq R_{t+1}$ by \Cref{lem:C-grow}.

So, let us assume that $\|x^{(t)}_{T_t} - x^*_{T_t}\|_\infty \le 2mn^2\kappa_A^2 \left\|x^*_{T_t}\right\|_\infty$, that is,  $\|x^{(t)}_{T_t}\|_\infty \le (2mn^2\kappa_A^2+1) \left\|x^*_{T_t}\right\|_\infty$.
We may also assume that $\|x^{(t)}_N\|_1>0$, as otherwise $x^{(t)} = x^*$.
By \Cref{lem:geo-decay}, there is an iteration $r = t + O((n - m)\log(n + \kappa_A))$ such that $n^2\kappa_A (2mn^2\kappa_A^2 + 1) \|x_N^{(r)}\|_1 < \|x_N^{(t)}\|_1$.
Hence,
\begin{equation*}
  (2mn^2\kappa_A^2 + 1)\|x_{T_t}^*\|_\infty \ge \|x_{T_t}^{(t)}\|_\infty
\ge \|x_{N}^{(t)}\|_\infty \ge \frac{1}{n}\|x_{N}^{(t)}\|_1 > n\kappa_A(2mn^2\kappa_A^2 + 1)\|x_{N}^{(r)}\|_1,
\end{equation*}
where the second inequality is due to $N\subseteq T_t$ by \Cref{lem:I-monotone}.
Thus, $\|x_{T_t}^*\|_\infty > n\kappa_A \|x_{N}^{(r)}\|_1$ and so $L_t \subsetneq L_r$.
\end{proof}

%% file: capacitated.tex
\section{Circuit Diameter Bound for the Capacitated Case}\label{sec:capacity}

In this section we consider diameter bounds for systems of the form
\begin{equation}\label{sys:bounded-polytope}\tag{Cap-P}
    P_u =\{x\in \R^n:\, Ax=b, \0 \le x \le u\}.
\end{equation}

The theory in \Cref{sec:diam-bound} carries over to $P_u$ at the cost of turning $m$ into $n$ via the standard reformulation 
\begin{equation}
    \label{eq:reformulation-bounded-polytope}
    \widetilde P_u = \left\{(x,y) \in \R^{n + n}:\, \begin{bmatrix}A & 0 \\ I & I\end{bmatrix} \begin{bmatrix} x \\ y \end{bmatrix} = \begin{bmatrix} b \\ u \end{bmatrix}, x, y \ge 0\right\}, \quad P_u = \{x : (x, y ) \in \widetilde P_u\}.
\end{equation}
\begin{corollary}\label{cor:bound-embed}
The circuit diameter of a system in the form \eqref{sys:bounded-polytope} with constraint matrix $A\in\R^{m\times n}$ is $O(n^2\log(n+\kappa_A))$.
\end{corollary}
\begin{proof}
Follows straightforward from \Cref{thm:circuit-diameter} together with the reformulation \eqref{eq:reformulation-bounded-polytope}. 
Let $\widetilde{A}$ denote the constraint matrix of \eqref{eq:reformulation-bounded-polytope}.
It is easy to check that $\kappa_A = \kappa_{\widetilde{A}}$, and that there is a one-to-one mapping between the circuits and maximal circuit augmentations of the two systems.
\end{proof}
Intuitively, the polyhedron should not become more complex; related theory in \cite{Brand2021} also shows how two-sided bounds can be incorporated in a linear program without significantly changing the complexity of solving the program.

\Cref{thm:circuit-diameter-bounds} is proved using a new procedure, which we outline below.
A basic feasible point $x^* \in P_u$ is characterised by a partition $B \cup L \cup H = [n]$ where $A_B$ is a basis (has full column rank), $x^*_L = \0_L$ and $x^*_H = u_H$. In $O(n \log n)$ iterations, we fix all but $2m$ variables to the same bound as in $x^*$; for the remaining system with $2m$ variables, we can use the standard reformulation. 

\Cref{proc:feas-diam-bounded} starts with a preprocessing. We let $S_t\subseteq L\cup H$ denote the set of indices where $x_i^{(t)}\neq x^*_i$, i.e., we are not yet at the required lower and upper bound. If $|S_t|\le m$, then we remove the indices in $(L\cup H)\setminus S_t$, and use the diameter bound resulting from the standard embedding as in Corollary~\ref{cor:bound-embed}. 

As long as $|S_t|>m$, we proceed as follows. We define the cost function $c\in \R^n$ by $c_i=0$ for $i\in B$, $c_i=1/u_i$ for $i\in L$, and $c_i=-1/u_i$ for $i\in H$. For this choice, we see that the optimal solution of the LP $\min_{x\in P_u}\pr{c}{x}$ is $x^*$ with optimal value $\pr{c}{x^*} = -|H|$.

Depending on the value of $\pr{c}{x^{(t)}}$, we perform one of two updates. As long as $\pr{c}{x^{(t)}}\ge -|H|+1$, we take a conformal decomposition of $x^*-x^{(t)}$, and pick the most improving augmenting direction from the decomposition. If $\pr{c}{x^{(t)}}<-|H|+1$, then we use a support circuit augmentation obtained from \textsc{Support-Circuit}$(A,c,x^{(t)},S_t)$.

Let us show that whenever \textsc{Support-Circuit} is called, $g^{(t)}$ is guaranteed to exist. This is because $|S_t|>m$ and $x^{(t)}_i>0$ for all $i\in S_t$. Indeed, if $x^{(t)}_j = 0$ for some $j\in S_t$, then $j\in H$ from the definition of $S_t$. However, this implies that
\[\langle c, x^{(t)} \rangle \geq \sum_{i\in H\setminus \{j\}} c_i x^{(t)}_i \geq -|H|+1,\]
which is a contradiction.

The cost $\pr{c}{x^{(t)}}$ is monotone decreasing, and it is easy to see that $\pr{c}{x^{(0)}}\le n$ for any initial solution $x^{(0)}$. Hence, within $O((n-m)\log n)$ iterations we must reach $\pr{c}{x^{(t)}}<-|H|+1$. Each support circuit augmentation sets $x_i^{(t+1)}=0$ for $i\in L$ or $x_i^{(t+1)}=u_i$ for $i\in H$; hence, we perform at most $n-m$ such augmentations. The formal proof is given below.

\begin{algorithm}[htb!]
  \caption{\textsc{Capacitated-Diameter-Bound}}
  \label{proc:feas-diam-bounded}
  \SetKwInOut{Input}{Input}
  \SetKwInOut{Output}{Output}
  \SetKwComment{Comment}{$\triangleright$\ }{}
  \SetKw{And}{\textbf{and}}
  \SetKw{Or}{\textbf{or}}
  \Input{Polyhedron in the form \eqref{sys:bounded-polytope}, partition $B\cup L \cup H = [n]$ with its corresponding vertex $x^* = (A^{-1}_Bb, \0_L, u_H)$, and initial vertex $x^{(0)}$.}
  \Output{Length of a circuit walk from $x^{(0)}$ to $x^*$.}
  Set the cost $c\in \R^n$ as $c_i = 0$ if $i\in B$, $c_i = 1/u_i$ if $i\in L$, and $c_i = -1/u_i$ if $i\in H$\;
  $t\gets 0$\;
  $S_0 \gets \{i\in L\cup H: x^{(0)}_i \neq x^*_i\}$\;
  \While{$|S_t|> m$}{
    \uIf{$\langle c,x^{(t)}\rangle \geq -|H|+1$}{
        Let $h^{(1)},h^{(2)},\dots,h^{(k)}$ be a conformal circuit decomposition of $x^* - x^{(t)}$\;
        $g^{(t)}\gets h^{(j)}$ for any $j\in \argmin_{i\in [k]}\langle c,h^{(i)} \rangle$\;
    }
    \Else{$g^{(t)}\gets \textsc{Support-Circuit}(A,c,x^{(t)},S_t)$}
    $x^{(t+1)} \gets \aug_P(x^{(t)}, g^{(t)})$\;
    $S_{t+1} \gets \{i\in L\cup H: x^{(t+1)}_i \neq x^*_i\}$; $t\gets t+1$  \; 
    }
  Run \Cref{proc:diam-bound} on $\widetilde A := \begin{bmatrix} A_{B \cup S_t} & 0 \\ I & I \end{bmatrix}$ and $\widetilde b = \begin{bmatrix}
       b \\ u \\
    \end{bmatrix}$ to get $t'\in \Z_+$\;
  \Return $t+t'$ \;
\end{algorithm}

\begin{proof}[Proof of \Cref{thm:circuit-diameter-bounds}]
We show that \Cref{proc:feas-diam-bounded} has the
claimed number of iterations. As previously mentioned, $\pr{c}{x^*} = -|H|$ is the
optimal value of the LP $\min_{x\in P_u}\pr{c}{x}$. Initially, $\pr{c}{x^{(0)}} = -\sum_{i\in H}
\frac{x^{(0)}}{u_i} +\sum_{i\in L} \frac{x^{(0)}}{u_i} \le n$. Similar to
\Cref{lem:geo-decay}, due to our choice of $g^{(t)}$ from the conformal
circuit decomposition, we have $\pr{c}{x^{(t+1)}} + |H| \le (1 -
\frac{1}{n-m})(\pr{c}{x^{(t)}} + |H|)$. In particular, $O((n-m)\log n)$ iterations
suffice to find an iterate $t$ such that $\langle c, x^{(t)} \rangle < - |H|
+ 1$. 

Note that the calls to \textsc{Support-Circuit} do not increase
$\pr{c}{x^{(t)}}$, so from now we will never make use of the conformal
circuit decomposition again. An augmentation resulting from a call to \textsc{Support-Circuit} will set at
least one variable $i \in \supp(g^{(t)})$  to either 0 or $u_i$.  We claim
that either $x^{(t+1)}_i=0$ for some $i\in L$, or $x^{(t+1)}_i=u_i$ for some
$i\in H$, that is, we set a variable to the `correct' boundary. To see this,
note that if $x^{(t+1)}_i$ hits the wrong boundary, then the gap between
$\pr{c}{x^{(t+1)}}$ and $-|H|$ must be at least $1$, a
clear contradiction to $\pr{c}{x^{(t+1)}} < -|H|+1$. 

Thus, after at most $n-m$ calls to  \textsc{Support-Circuit},  we get $|S_{t}|
\le m$, at which point we call \Cref{proc:diam-bound} with at most $2m$
variables, so the diameter bound of \Cref{thm:circuit-diameter} applies.
\end{proof}

%% file: feasibility.tex
\section{Proximity Results}
\label{sec:proximity}
We now present Hoffman-proximity bounds in terms of the circuit
 imbalance measure $\kappa_A$. A simple such bound was Lemma~\ref{lem:circuit-N-kappa}; we now present additional norm bounds.
   These can be derived from more general results in \cite{DadushNV20}; see also \cite{Ekbatani2021}. The references also explain the background and similar results in previous literature, in particular, to proximity bounds via $\Delta_A$ in e.g., \cite{Tardos86} and \cite{Cook1986}. For completeness, we include the proofs.

The next technical lemma will be key in our arguments. See \Cref{cor:basic-norm-bound} below for a simple implication. 
\begin{lemma}\label{lem:circuit-T-indep}
Let $A\in \R^{m\times n}$ and $x\in \R^n$.
Let $L\subseteq \supp(x)$ and $S\subseteq [n]\setminus L$.
If there is no circuit $C\subseteq  \supp(x)$ such that $C\cap S\neq \emptyset$, then
\[\|x_S\|_\infty \leq \kappa_A\min_{z\in \ker(A)+x}\|z_{[n]\setminus \cl(L)}\|_1 \,.\] %
\end{lemma} 
Before the proof, it is worth stating a useful special case $L = \emptyset$ and $S=[n]$.%

\begin{corollary}\label{cor:basic-norm-bound}
Let $x$ be a basic (but not necessarily feasible) solution to \eqref{sys:lp}. Then, for any $z$ where $Az=b$, we have
$\|x\|_\infty\leq \kappa_A\|z\|_1 $. 
\end{corollary}

\begin{proof}[Proof of Lemma~\ref{lem:circuit-T-indep}]
First, we show that $x_{S\cap \cl(L)} = \0$ due to our assumption.
Indeed, any $i\in S\cap \cl(L)$ with $x_i\neq 0$ gives rise to a circuit in $L\cup \{i\}\subseteq \supp(x)$, contradicting the assumption in the lemma.
It follows that $\|x_S\|_\infty = \|x_{S\setminus \cl(L)}\|_\infty$;
let $j\in S\setminus \cl(L)$ such that $|x_j| = \|x_{S}\|_\infty$.
Let $z\in \ker(A)+x$ be a minimizer of the RHS in the statement.
We may assume that $|x_j|> |z_j|$, as otherwise we are done because $\kappa_A\geq 1$.

Let $h^{(1)},\ldots, h^{(k)}$ be a conformal circuit decomposition of $z - x \in \ker(A)$. 
Among these elementary vectors, consider the set $R := \{t\in [k]: h^{(t)}_j\neq 0\}$.
\begin{claim}\label{clm:circuit-T-indep}
For each $t\in R$, there exists an index $i(t)\in \supp(h^{(t)})\setminus \cl(L)$ such that $x_{i(t)}= 0$ and $z_{i(t)}\neq 0$.
\end{claim}
\begin{proof}
For the purpose of contradiction, suppose that $\supp(h^{(t)})\setminus \cl(L)\subseteq \supp(x)$.
For every $i\in \cl(L)\setminus L$, we can write $A_i = Ay^{(i)}$ where $\supp(y^{(i)})\subseteq L$.
Consider the vector 
\[h\coloneqq h^{(t)} + \sum_{i\in \cl(L)\setminus L} h^{(t)}_i (y^{(i)} - e_i).\]
Clearly, $h_{\cl(L)\setminus L} = \0$ and $h_{[n]\setminus \cl(L)} = h^{(t)}_{[n]\setminus \cl(L)}$.
Since $L\subseteq \supp(x)$ and we assumed $\supp(h^{(t)})\setminus \cl(L)\subseteq \supp(x)$, it follows that $\supp(h)\subseteq \supp(x)$.
Moreover, $j\in \supp(h)$ because $j\in S\setminus \cl(L)$. 
Hence, applying \Cref{lem:conformal} to $h\in \ker(A)$ yields an elementary vector $g\in \EE(A)$ such that $\supp(g)\subseteq \supp(x)$ and $\supp(g)\cap S\neq \emptyset$.
This contradicts the assumption of the lemma.
\end{proof}

By conformality of the decomposition, $|x_j-z_j|=\sum_{t\in R} |h^{(t)}_j|$.
According to \Cref{clm:circuit-T-indep}, for every $t\in R$, we have $|h^{(t)}_j|\le \kappa_A |h^{(t)}_{i(t)}|$ where $i(t)\in [n]\setminus (\cl(L)\cup \{j\})$; notice that $i(t)\neq j$ for all $t\in R$ due to our assumption $|x_j|>0$.
Applying conformality again yields
\[\sum_{t\in R}|h^{(t)}_{i(t)}| = \sum_{s\in [n]\setminus (\cl(L)\cup\{j\})} \sum_{i(t) = s} |h^{(t)}_{i(t)}| \leq \sum_{s\in [n]\setminus (\cl(L)\cup\{j\})} |z_s| = \|z_{[n]\setminus (\cl(L)\cup\{j\})}\|_1.\]
Therefore,
\[\|x_S\|_\infty  = |x_j| \leq |z_j| + |x_j - z_j| \leq \kappa_A\|z_{[n]\setminus \cl(L)}\|_1\]
where the last inequality is obtained by combining the previous equation and inequalities.
\end{proof}

The following proximity theorem will be key to derive $x^*_i=0$ for certain variables in our optimization algorithm;
see \cite{DadushNV20} and \cite[Theorem 6.5]{Ekbatani2021}. 
For
$\tilde{c}\in \R^n$, we use $\LP(\tilde{c})$ to denote \eqref{sys:lp} with cost vector $\tilde{c}$, and $\OPT(\tilde{c})$ as the optimal value of $\LP(\tilde{c})$.

\begin{theorem}\label{thm:fixing}
Let $c,c'\in \R^n$ be two cost vectors, such that both $\LP(c)$ and $\LP(c')$ have finite optimal values. Let $s'$ be a dual optimal solution to $\LP(c')$. 
For all indices $j\in [n]$ such that
\[s'_j> (m+1)\kappa_A \|c-c'\|_\infty\,,\]
it follows that $x^*_j=0$ for every optimal solution $x^*$ to $\LP(c)$.
\end{theorem}

\begin{proof}
We may assume that $c\neq c'$, as otherwise we are done by complementary slackness.
Let $x'$ be an optimal solution to $\LP(c')$. By complementary slackness, $s'_jx'_j=0$, and therefore
 $x'_j=0$.
For the purpose of contradiction, suppose that there exists an optimal solution $x^*$ to $\LP(c)$ such that $x^*_j>0$.
Let $h^{(1)},\ldots, h^{(k)}$ be a conformal 
 circuit decomposition of $x^*-x'$.
Then, $h^{(t)}_j>0$ for some $t\in [k]$.
Since $h^{(t)}$ is an elementary vector, $|\supp{(h^{(t)})}|\leq m+1$ and so $\|h^{(t)}\|_1\le (m+1)\|h^{(t)}\|_\infty\le(m+1)\kappa_A h_j^{(t)}$.
Observe that for any $i\in [n]$ where $h^{(t)}_i < 0$, we have $s'_i = 0$ because $x'_i >  x^*_i \geq 0$.
Hence,
\begin{align*}
  \pr{c}{h^{(t)}} &= \pr{c-c'}{h^{(t)}} + \pr{c'}{h^{(t)}} \ge - \|c-c'\|_\infty \|h^{(t)}\|_1 + \pr{s'}{h^{(t)}} \\
  &\geq -(m+1)\kappa_A\|c-c'\|_\infty\,h^{(t)}_j + s'_jh^{(t)}_j > 0\, .
\end{align*}
The first inequality here used H\"older's inequality and that $\pr{c'}{h^{(t)}}=\pr{s'}{h^{(t)}}$ since $c'-s'\in \im(A^\top)$ and $h^{(t)}\in \ker(A)$.
Since $x^*-h^{(t)}$ is feasible to $\LP(c)$, this contradicts the optimality of $x^*$.
\end{proof}

The following lemma provides an upper bound on the norm of the perturbation $c-c'$ for which the existence of an index $j$ as in \Cref{thm:fixing} is guaranteed.

\begin{lemma}\label{lem:big-slack}
Let $c,c'\in \R^n$ be two cost vectors, and let $s'$ be an optimal dual solution to $\LP(c')$.
If $c\in \ker(A)$, $\|c\|_2 = 1$ and $\|c-c'\|_\infty <  1/(\sqrt{n}(m+ 2)\kappa_A)$, then there exists an index $j\in [n]$ such that 
\[s'_j>\frac{m+1}{\sqrt{n}(m+2)}.\]
\end{lemma}

\begin{proof}
Let $r = c-c'$.
Note that $s'+r\in \im(A^{\top})+c$.
Then,
\[\|s'\|_\infty + \|r\|_\infty \geq\|s'+r\|_\infty\geq \frac{1}{\sqrt{n}}\|s'+r\|_2\geq \frac{1}{\sqrt{n}}\|c\|_2 = \frac{1}{\sqrt{n}},\]
where the last inequality is due to $s'+r-c$ and $c$ being orthogonal.
This gives us
\[\|s'\|_\infty \geq \frac{1}{\sqrt{n}} - \|r\|_\infty > \frac{(m+2)\kappa_A-1}{\sqrt{n}(m+2)\kappa_A} \geq \frac{m+1}{\sqrt{n}(m+2)}\]
as desired because $\kappa_A\geq 1$.
\end{proof}

\section{A Circuit Augmentation Algorithm for Feasibility}\label{sec:feas}

In this section we prove Theorem~\ref{thm:feasible-augment}: given a linear program \eqref{sys:lp} 
with cost $c = (\0_{[n]\setminus N}, \1_N)$ for some $N\subseteq [n]$, find a solution $x$ with $x_N=\0$ (showing that the optimum value is 0), or certify that no such solution exists.
A dual certificate in the latter case is a vector $y\in\R^m$ such that $A^\top y\le c$ and $\pr{b}{y}>0$.

\Cref{thm:feasible-augment} can be used to solve the feasibility problem for linear programs. Given a polyhedron in standard form \eqref{sys:polytope}, we construct an auxiliary linear program whose feasibility problem is trivial, and whose optimal solutions correspond to feasible solutions to \eqref{sys:polytope}. This is in the same tune as Phase I of the Simplex method:
\begin{equation}
    \label{sys:aux-lp}\tag{Aux-LP}
    \min \; \pr{\1}{z}  \quad \mathrm{s.t.}\quad
    Ay - Az=b\, , \, 
    y, z \geq 0\, .\\
\end{equation}
For the constraint matrix $\widetilde A = \begin{bmatrix} A & - A\end{bmatrix}$, it is easy to see that $\kappa_{\widetilde A} = \kappa_A$ and that any solution $Ax = b$ can be converted into a feasible solution to \eqref{sys:aux-lp} via $(y,z) = (x^+, x^-)$. Hence, if 
the subroutines \textsc{Support-Circuit} and \textsc{Ratio-Circuit} are available for \eqref{sys:aux-lp}, then we can invoke \Cref{thm:feasible-augment} with $N = \{n+1,n+2,\dots,2n\}$ on \eqref{sys:aux-lp} to solve the feasibility problem of \eqref{sys:polytope} in $O(mn\log(n+\kappa_A))$ augmentation steps.

\medskip
Our algorithm is presented in \Cref{alg:feas}. We maintain a set $\cL_t\subseteq [n]\setminus N$, initialized as $\emptyset$.  
Whenever $x^{(t)}_i \ge 4mn\kappa_A \|x_N^{(t)}\|_1$ for the current iterate $x^{(t)}$, we add $i$ to $\cL_t$. 
Note that once an index $i$ enters $\mathcal{L}_t$, it is never removed, even though $x_i$ might drop below this threshold in the future. Still, we will show that $\mathcal{L}_t\subseteq \supp(x^{(t)})$ in every iteration.

Whenever $\rk(\cL_t)$ increases, we run  \textsc{Support-Circuit}$(A,c,x^{(t)},N)$ iterations
 as long as there exists a circuit in $\supp(x^{(t)})$ intersecting $N$. Afterwards, we run a sequence of \textsc{Ratio-Circuit} iterations until $\rk(\cL_t)$ increases again.
The key part of the analysis is to show that $\rk(\cL_t)$ increases in every $O(n\log (n+\kappa_A))$ iterations.

\begin{algorithm}[htb!]
  \caption{\textsc{Feasibility-Algorithm}}
  \label{alg:feas}
  \SetKwInOut{Input}{Input}
  \SetKwInOut{Output}{Output}
  \SetKwComment{Comment}{$\triangleright$\ }{}
  \SetKw{And}{\textbf{and}}
  \SetKw{Or}{\textbf{or}}
  \Input{Linear program in standard form \eqref{sys:lp} with cost $c = (\0_{[n]\setminus N}, \1_N)$ for some $N\subseteq [n]$,  and initial feasible solution $x^{(0)}$.}
  \Output{A solution $x$ with $x_N=\0$, or a dual solution $y\in\R^m$, $A^\top y\le c$, $\pr{b}{y}>0$.}
 $t\gets 0$ ; $\cL_{t-1} \gets \emptyset$ \;
  \While{$x^{(t)}_N\neq\0$}{
    $\cL_t\gets \cL_{t-1}\cup \{i\in [n]:\, x^{(t)}_i \ge 4mn\kappa_A \|x_N^{(t)}\|_1\}$ \; 
    \If{$t=0$ \Or $\rk(\cL_t)> \rk(\cL_{t-1})$}{
    \While{$\exists$ a circuit in $\supp(x^{(t)})$ intersecting $N$}{ %
    $g^{(t)}\gets \textsc{Support-Circuit}(A,c,x^{(t)},N)$ \; %
    $x^{(t+1)} \gets \aug_P(x^{(t)}, g^{(t)})$ \; 
    $\cL_{t+1}\gets \cL_t$ ; $t\gets t+1$  \; 
    }
    \If{$x^{(t)}_N = \0$}{
      \Return $x^{(t)}$ \; %
    }
    }
    $(g^{(t)},y^{(t)},s^{(t)})\gets \textsc{Ratio-Circuit}(A,c,1/x^{(t)})$ \;
    \If{$\pr{b}{y^{(t)}}>0$}{Terminate with infeasibility certificate ;}
    $x^{(t+1)} \gets \aug_P(x^{(t)}, g^{(t)})$; $t\gets t+1$  \; 
    }
  \Return $x^{(t)}$ \;
\end{algorithm}

Let us first analyze what happens during \textsc{Ratio-Circuit} iterations.

\begin{lemma}\label{cl:ratio}
If  \textsc{Ratio-Circuit} is called in iteration $t$, then either $\|x_N^{(t+1)}\|_1\le \left(1-\frac1n\right)\|x_N^{(t)}\|_1$, or the algorithm terminates with a dual certificate.
\end{lemma}
\begin{proof}
 The oracle returns $g^{(t)}$ that is optimal to \eqref{sys:minratio} and $(y^{(t)},s^{(t)})$ that is optimal to \eqref{sys:minratio-dual} with optimum value $-\lambda$. Thus, $A^\top y+s=c$ and $\0\le s\le \lambda w$. Recall that we use weights $w_i=1/x^{(t)}_i$.
 If $\pr{b}{y^{(t)}}>0$, the algorithm terminates. Otherwise,
note that  
\[
 \pr{c}{x^{(t)}}=\pr{b}{y^{(t)}}+\pr{s^{(t)}}{x^{(t)}}\le 
 \lambda \pr{w_{\supp(x^{(t)})}}{x^{(t)}_{\supp(x^{(t)})}}\leq n \lambda\, ,
 \]
 implying $\lambda\ge \pr{c}{x^{(t)}}/n$, and therefore 
 $\pr{c}{g^{(t)}}=-\lambda\le -\pr{c}{x^{(t)}}/n$. This implies the lemma, noting that
 \[
 \|x_N^{(t+1)}\|_1=\pr{c}{x^{(t+1)}}\le \pr{c}{x^{(t)}}+\pr{c}{g^{(t)}}\le \left(1-\frac1n\right)\|x_N^{(t)}\|_1\, .
 \qedhere
 \]
\end{proof}

Next, we analyze what happens during \textsc{Support-Circuit} iterations.

\begin{lemma}\label{cl:support}
If \textsc{Support-Circuit} is called in iteration $t$, then $\|x^{(t+1)}-x^{(t)}\|_\infty \le \kappa_A\|x_N^{(t)}\|_1$.
\end{lemma}

\begin{proof}
We have $g^{(t)}_i < 0$ for some $i\in N$ because $\supp(g^{(t)})\cap N\neq \emptyset$ and $\pr{c}{g^{(t)}}\leq 0$.
Hence,
\[\|x^{(t+1)}-x^{(t)}\|_\infty\le  \kappa_A |x^{(t+1)}_i  - x^{(t)}_i|\leq \kappa_A x^{(t)}_i \leq \kappa_A\|x_N^{(t)}\|_1. \qedhere\]
\end{proof}
The following lemma shows that once a coordinate enters $\cL_t$, its value stays above a certain threshold.

\begin{lemma}\label{cl:clt}
For every iteration $t\geq 0$, we have $x_j^{(t)}\ge 2mn\kappa_A\|x^{(t)}_N\|_1$ for all $j\in \cL_t$.
\end{lemma}

\begin{proof}
Fix an iteration $t\geq 0$ and a coordinate $j\in \cL_t$.
We may assume that $\|x^{(t)}_N\|_1 > 0$, as otherwise the lemma trivially holds because $x^{(t)}\geq \0$.
Let $r\le t$ be the iteration in which $j$ was added to $\cL_r$; the lemma clearly holds at iteration $r$.

We analyze the ratio $x^{(t')}_j/\|x^{(t')}_N\|_1$ for iterations $t'=r,\ldots,t$.  
At an iteration $r\leq t'<t$ that performs \textsc{Ratio-Circuit}, observe that if $x^{(t')}_j/\|x^{(t')}_N\|_1\ge 2n\kappa_A$, then 
    \begin{equation*}
      \label{eq:monotone_ratio_circuit}
      \begin{aligned}
    \frac{x^{(t'+1)}_j}{\|x_N^{(t'+1)}\|_1} \ge \frac{x^{(t')}_j - \kappa_A\|x_N^{(t'+1)} - x_N^{(t')}\|_1}{(1- \frac{1}{n})\|x_N^{(t')}\|_1} \ge \frac{x^{(t')}_j - 2\kappa_A\|x_N^{(t')}\|_1}{(1- \frac{1}{n})\|x_N^{(t')}\|_1} \ge \frac{(1-\frac{1}{n})x_j^{(t')}}{(1 - \frac{1}{n})\|x_N^{(t')}\|_1} = \frac{x^{(t')}_j}{\|x_N^{(t')}\|_1} \,. 
    \end{aligned}
    \end{equation*}
The first inequality is due to \Cref{cl:ratio} and the fact that $x^{(t'+1)} - x^{(t')}$ is an elementary vector whose support intersects $N$. This fact follows from $\langle c, g^{(t')} \rangle < 0$ because $\|x^{(t')}_N\|_1 \geq \|x^{(t)}_N\|_1 > 0$ and $\langle b, y^{(t')} \rangle\leq 0$.
The second inequality uses the monotonicity $\|x^{(t'+1)}_N\|_1\le \|x^{(t')}_N\|_1$ and the triangle inequality. The third inequality uses the assumption $x^{(t')}_j/\|x^{(t')}_N\|_1\ge 2n\kappa_A$.

Hence, it suffices to show that \textsc{Support-Circuit} maintains the invariant $x^{(t')}_j/\|x^{(t')}_N\|_1\ge 2n\kappa_A$. At an iteration $r\leq t' < t$ which performs \textsc{Support-Circuit}, we have
\[\frac{x^{(t'+1)}_j}{\|x^{(t'+1)}_N\|_1} \geq \frac{x^{(t')}_j - \kappa_A\|x^{(t')}_N\|_1}{\|x^{(t')}_N\|_1} = \frac{x^{(t')}_j}{\|x^{(t')}_N\|_1} - \kappa_A\]
by \Cref{cl:support}. Since \Cref{alg:feas} performs at most $(m+1)n$ \textsc{Support-Circuit} iterations, the total decrease of this ratio is at most $(m+1)n\kappa_A \leq 2mn\kappa_A$.
As the starting value is at least $4mn\kappa_A$, it follows that this ratio does not drop below $2mn\kappa_A$.
\end{proof}

\smallskip
\begin{proof}[Proof of Theorem~\ref{thm:feasible-augment}]
The correctness of \Cref{alg:feas} is obvious.
If the algorithm terminates due to $x^{(t)}_N=0$, then $x^{(t)}$ is the desired solution to \eqref{sys:lp}.
Otherwise, if the algorithm terminates due to $\langle b, y^{(t)} \rangle>0$, then $y^{(t)}$ is the desired dual certificate as it is feasible to \eqref{sys:dual_lp}.

Next, we show that if $\rk(\cL_t) = m$, then the algorithm will terminate in iteration $r \leq t+ n$ with $x^{(r)}_N = \0$.
As long as $x^{(t)}_N \neq \0$, we have $\cL_t\subseteq [n]\setminus N$ by \Cref{cl:clt}.
Moreover, any $i\in \supp(x^{(t)}_N)$ induces a circuit in $\cL_t\cup \{i\}$, so \textsc{Support-Circuit} will be invoked.
Since every call to \textsc{Support-Circuit} reduces $\supp(x^{(t)})$, all the coordinates in $N$ will be zeroed-out in at most $n$ calls.

It is left to bound the number of iterations of \Cref{alg:feas}.
In the first iteration and whenever $\rk(\cL_t)$ increases, we perform a sequence of at most $n$ \textsc{Support Circuit} cancellations. Let us consider an iteration $t$ right after we are done with the \textsc{Support Circuit} cancellations. Then, there is no circuit in $\supp(x^{(t)})$ intersecting $N$. We show that $\rk(\cL_t)$ increases within $O(n\log(n+\kappa_A))$ consecutive calls to \textsc{Ratio-Circuit}; this completes the proof.

By \Cref{cl:ratio}, within $O(n\log (n\kappa_A)) = O(n \log(n + \kappa_A))$ consecutive \textsc{Ratio-Circuit} augmentations, we reach an iterate $r = t + O(n \log(n + \kappa_A))$ such that $\|x_N^{(r)}\|_1 \le (4 mn^3\kappa_A^2)^{-1} \|x_N^{(t)}\|_1$.
Since $\cL_t\subseteq \supp(x^{(t)})$ and $N\subseteq [n]\setminus \cL_t$ by \Cref{cl:clt}, and there is no circuit in $\supp(x^{(t)})$ intersecting $N$, applying Lemma~\ref{lem:circuit-T-indep} with $x=x^{(t)}$ and $z=x^{(r)}$ yields
\[
\|x^{(r)}_{[n]\setminus \cl(\cL_t)}\|_\infty \ge \frac{\|x^{(r)}_{[n]\setminus \cl(\cL_t)}\|_1}{n}\ge \frac{\|x^{(t)}_{N}\|_\infty}{n\kappa_A} \geq \frac{\|x^{(t)}_{N}\|_1}{n^2\kappa_A}\geq {4mn\kappa_A}{\|x_N^{(r)}\|_1}\, ,
\]
showing that some $j\in [n]\setminus \cl(\cL_t)$ must be included in $\cL_r$.
    \end{proof}

%% file: optimization.tex
\section{A Circuit Augmentation Algorithm for Optimization}\label{sec:opt}

In this section, we give a circuit-augmentation algorithm for solving \eqref{sys:lp}, given by $A\in\R^{m\times n}$, $b\in \R^m$ and $c\in\R^n$. We also assume that an initial feasible solution $x^{(0)}$ is provided. 
In every iteration $t$, the algorithm maintains a feasible solution $x^{(t)}$ to \eqref{sys:lp}, initialized with $x^{(0)}$.
The goal is to augment $x^{(t)}$ using the subroutines \textsc{Support-Circuit} and \textsc{Ratio-Circuit} until the emergence of a nonempty set $N\subseteq [n]$ which satisfies $x^{(t)}_N=x^*_N=\0$ for every optimal solution $x^*$ to \eqref{sys:lp}. 
When this happens, we have reached a lower dimensional face of the polyhedron that contains the optimal face.
Hence, we can fix $x^{(t')}_N=\0$ in all subsequent iterations $t'\geq t$. 
In particular, we repeat the same procedure on a smaller LP with constraint matrix $A_{[n]\setminus N}$, RHS vector $b$, and cost $c_{[n]\setminus N}$, initialized with the feasible solution $x^{(t)}_{[n]\setminus N}$.
Note that a circuit walk of this smaller LP corresponds to %
a circuit walk of the original LP.
This gives the overall circuit-augmentation algorithm. 

\medskip

In what follows, we focus on the aforementioned variable fixing procedure (Algorithm \ref{alg:circ_aug_opt}), since the main algorithm just calls it at most $n$ times.

We fix parameters
\[\delta \coloneqq \frac{1}{2n^{3/2}(m+2)\kappa_A}\, , \qquad T\coloneqq \Theta(n\log(n+\kappa_A))\, , \qquad \Gamma\coloneqq \frac{6(m+2)\sqrt{n}\kappa_A^2T}{\delta}\, .\] 
Throughout the procedure, $A$ and $b$ will be fixed, but we will sometimes modify the cost function $c$. %
Recall that for any $\tilde{c}\in \R^n$, we use $\LP(\tilde{c})$ to denote the problem with cost vector $\tilde{c}$, and the optimal value is $\OPT(\tilde{c})$. We will often use the fact that if $\tilde s\in \im(A^\top)+\tilde c$, then the linear programs $\LP(\tilde s)$ and $\LP(\tilde c)$ are equivalent.

\medskip

Let us start with a high level overview before presenting the algorithm.
The inference that $x^{(t)}_N=x^*_N=\0$ for every optimal $x^*$ will be made using Theorem~\ref{thm:fixing}. To apply this, our goal is to find a cost function $c'$ and an optimal dual solution $s'$ to $\LP(c')$ such that the set of indices $N\coloneqq \{\,j: s'_j> (m+1)\kappa_A \|c-c'\|_\infty\,\}$ is nonempty.

If $c=\0$, then we can return $x^{(0)}$ as an optimal solution. Otherwise, we can normalize to $\|c\|=1$.\footnote{Taking square roots can be avoided by normalizing with $\|c\|_1 = 1$ or $\|c\|_\infty = 1$ instead, and changing the parameters of the algorithm accordingly.} %
Let us start from any primal and dual feasible solutions $(x^{(0)},s^{(0)})$ to $\LP(c)$; we can obtain $s^{(0)}$ from a call to \textsc{Ratio-Circuit}. 
 Within $O(n\log(n+\kappa_A))$ \textsc{Ratio-Circuit} augmentations, we arrive at a pair of primal and dual feasible solutions $(x,s)=(x^{(t)},s^{(t)})$ such that $\pr{x}{s}\le\varepsilon\coloneqq  \pr{x^{(0)}}{s^{(0)}}/\mathrm{poly}(n,\kappa_A)$.

We now describe the high level motivation for the algorithm.
Suppose that for every $i\in\supp(x)$, $s_i$ is small, say $s_i< \delta$. Let $\tilde c_i\coloneqq s_i$ if $i\notin\supp(x)$ and $\tilde c_i\coloneqq 0$ if $i\in\supp(x)$. Then, $\|\tilde c-s\|_\infty< \delta$ and $x$ and $\tilde c$ are primal and dual optimal solutions to $\LP(\tilde c)$. This follows because they are primal and dual feasible and satisfy complementary slackness.
Consider the vector $c'\coloneqq c-s+\tilde c$, which satisfies $\|c-c'\|_\infty < \delta$. 
Since $c-s\in\im(A^\top)$, $\LP(c')$ and $\LP(\tilde c)$ are equivalent. Thus, $x$ and $\tilde c$ are primal and dual optimal solutions to $\LP(c')$. Then, Theorem~\ref{thm:fixing} is applicable for the costs $c, c'$ and the dual optimal solution $\tilde{c}$. However, to be able to make progress by fixing variables, we also need to guarantee that $N\neq \emptyset$. Following Tardos \cite{Tardos85,Tardos86}, this can be ensured if we pre-process by projecting the cost vector $c$ onto $\ker(A)$; this guarantees that $\|s\|$---and thus $\|\tilde c\|$---must be sufficiently large.

Let us now turn to the case when the above property does not hold
 for $(x,s)$: for certain coordinates we could have $x_i>0$ and $s_i\geq\delta$. 
We enter the second \emph{phase} of the algorithm. Let $S = \{i\in [n]:s_i\geq \delta\}$ be the coordinates with large dual slack.
 Since $x_is_i\le \pr{x}{s}\le\varepsilon$, this implies $x_i\leq\varepsilon/\delta$ for all $i\in S$.  Therefore, $\|x_S\|$ is sufficiently small, and one can show that the set of `large' indices $\cL=\{i\in [n]:\, x_i\geq \Gamma\|x_S\|_1\}$ is nonempty. We proceed by defining a new cost function $\tilde c_i\coloneqq s_i$ if $i\in S$ and $\tilde c_i\coloneqq 0$ if $i\notin S$. 
 We perform \textsc{Support-Circuit} iterations as long as there exist circuits in $\supp(x)$ intersecting $\supp(\tilde c)$, and then perform further  $O(n\log(n+\kappa_A))$ {\sc Ratio-Circuit} iterations for the cost function $\tilde c$. If we now arrive at an iterate $(x,s)=(x^{(t')},s^{(t')})$ such that 
 $s_i < \delta$ for every $i\in\supp(x)$, then we truncate $s$ as before to an optimal dual solution to $\LP(c'')$ for some vector $c''$ where $\|c-c''\|_\infty < 2\delta$.
 After that, Theorem~\ref{thm:fixing} is applicable for the costs $c,c''$ and said optimal dual solution. 
 Otherwise, we continue with additional phases.

 The algorithm formalizes the above idea, with some technical modifications. The algorithm comprises at most $m+1$ phases; the main potential is that the rank of the large index set $\cL$ increases in every phase. We show that if an index $i\notin \cl(\cL)$
 was added to $\cL$, then it must have $s_i<\delta$ at the beginning of every later phase. Thus, these indices cannot be violating anymore.

\begin{algorithm}[htb!]
  \caption{\textsc{Variable-Fixing}}
  \label{alg:circ_aug_opt}
  \SetKwInOut{Input}{Input}
  \SetKwInOut{Output}{Output}
  \SetKwComment{Comment}{$\triangleright$\ }{}
  \SetKw{And}{\textbf{and}}
  \SetKw{Or}{\textbf{or}}
  \Input{Linear program in standard form \eqref{sys:lp}, and initial feasible solution $x^{(0)}$.} 
  \Output{Either an optimal solution to \eqref{sys:lp}, or a feasible solution $x$ and $\emptyset\neq N\subseteq [n]$ such that $x_N = x^*_N = \0$ for every optimal solution $x^*$ to \eqref{sys:lp}.}
    $t\gets 0$; $k\gets 0$; $\cL_{t-1}\gets \emptyset$\;
    $c\gets \Pi_{\ker(A)}(c)$\;
    \If{$c=\0$}{
      \Return{$x^{(0)}$}\;
    }
    $c\gets c/\|c\|_2$\;
    $(\cdot,\cdot,\tilde{s}^{(-1)})\gets \textsc{Ratio-Circuit}(A,c,\1)$ \Comment*{Any dual feasible solution to $\LP(c)$}
    \While{$\pr{\tilde{s}^{(t-1)}}{x^{(t)}} > 0$}{
      $S_t\gets \{i\in [n]:\tilde{s}_i^{(t-1)} \geq \delta\}$\;
      $\cL_t\gets \cL_{t-1}\cup \{i\in [n]:\, x^{(t)}_i \ge \Gamma \|x_{S_t}^{(t)}\|_1\}$\;
      \If{$t=0$ \Or $\rk(\cL_t)> \rk(\cL_{t-1})$}{
        $k\gets k+1$ \Comment*{New phase}
        Set modified cost $\tilde{c}^{(k)}\in \R^n_+$ as $\tilde{c}^{(k)}_i \gets \tilde{s}^{(t-1)}_i$ if $i\in S_t$, and
         $\tilde{c}^{(k)}_i \gets 0$ otherwise\;
        \While{$\exists$ a circuit in $\supp(x^{(t)})$ intersecting $\supp(\tilde{c}^{(k)})$}{
          $g^{(t)}\gets \textsc{Support-Circuit}(A,\tilde{c}^{(k)},x^{(t)},\supp(\tilde{c}^{(k)}))$ \;
          $x^{(t+1)} \gets \aug_P(x^{(t)},g^{(t)})$\; %
          $\cL_{t+1}\gets \cL_t$; $t\gets t+1$\; %
        }
      }
      $(g^{(t)},y^{(t)},s^{(t)}) \gets \textsc{Ratio-Circuit}(A,\tilde{c}^{(k)},1/x^{(t)})$\;
      \uIf{$\langle \tilde{c}^{(k)}, g^{(t)} \rangle = 0$}{
       $x^{(t+1)} \gets x^{(t)}$ \Comment*{Terminating in the next iteration by \Cref{clm:terminate}}
      }
      \Else{
       $x^{(t+1)} \gets \aug_P(x^{(t)},g^{(t)})$\; 
      }
      $\tilde{s}^{(t)} \gets \argmin_{s\in \{\tilde{c}^{(k)},s^{(t)}\}}\pr{s}{x^{(t+1)}}$; $t\gets t+1$\; 
    }
    $N\gets \{i\in [n]:\tilde{s}^{(t-1)}_i > \kappa_A (m+1)n\delta\}$\; %
  \Return $(x^{(t)},N)$\;
\end{algorithm}

\medskip

We now turn to a more formal description of 
 Algorithm \ref{alg:circ_aug_opt}. We start by  orthogonally projecting the input cost vector $c$ to $\ker(A)$. 
This does not change the optimal face of \eqref{sys:lp}.
If $c=\0$, then we terminate and return the current feasible solution $x^{(0)}$ as it is optimal. %
Otherwise, we scale the cost to $\|c\|_2 = 1$, and use \textsc{Ratio-Circuit} to obtain a feasible solution $\tilde{s}^{(-1)}$ to the dual of $\LP(c)$. 

The rest of Algorithm \ref{alg:circ_aug_opt} consists of repeated \emph{phases}, ending when $\pr{\tilde{s}^{(t-1)}}{x^{(t)}} = 0$.
In an iteration $t$, let $S_t = \{i\in [n]:\tilde{s}^{(t-1)}_i \geq \delta\}$ be the set of coordinates with large dual slack.
The algorithm keeps track of the following set
\[\cL_t \coloneqq  \cL_{t-1} \cup\set{i\in [n]:x^{(t)}_i \geq \Gamma \|x^{(t)}_{S_t}\|_1}.\]
These are the variables that were once large with respect to $\|x^{(t')}_{S_{t'}}\|_1$ in iteration $t'\leq t$.
Note that $|\cL_t|$ is monotone nondecreasing.

The first phase starts at  $t=0$, and we enter a new phase $k$ whenever $\rk(\cL_t) > \rk(\cL_{t-1})$.
Such an iteration $t$ is called the \emph{first iteration} in phase $k$.
At the start of the phase, we define a new \emph{modified cost} $\tilde{c}^{(k)}$ from the dual slack $\tilde{s}^{(t-1)}$ by truncating entries less than $\delta$ to 0.
This cost vector will be used until the end of the phase. 
Then, we call \textsc{Support-Circuit}$(A,\tilde{c}^{(k)},x^{(t)},\supp(\tilde{c}^{(k)}))$ to eliminate circuits in $\supp(x^{(t)})$ intersecting $\supp(\tilde{c}^{(k)})$.
Note that there are at most $n$ such calls because each call sets a primal variable $x^{(t)}_i$ to zero.

In the remaining part of the phase, we augment $x^{(t)}$ using \textsc{Ratio-Circuit}$(A,\tilde{c}^{(k)},1/x^{(t)})$ until $\rk(\cL_t)$ increases, triggering a new phase.
In every iteration, \textsc{Ratio-Circuit}$(A,\tilde{c}^{(k)},1/x^{(t)})$ returns a minimum cost-to-weight ratio circuit $g^{(t)}$, where the choice of weights $1/x^{(t)}$ follows Wallacher \cite{Wallacher}.
It also returns a feasible solution $(y^{(t)},s^{(t)})$ to the dual of $\LP(\tilde{c}^{(k)})$.
After augmenting $x^{(t)}$ to $x^{(t+1)}$ using $g^{(t)}$, we update the dual slack as 
$$\tilde{s}^{(t)} \coloneqq  \argmin_{s\in \{\tilde{c}^{(k)}, s^{(t)}\}} \pr{s}{x^{(t+1)}}.$$
This finishes the description of a phase.

Since $\rk(A) = m$, clearly there are at most $m+1$ phases.
Let $k$ and $t$ be the final phase and iteration of \Cref{alg:circ_aug_opt} respectively.
As $\pr{\tilde{s}^{(t-1)}}{x^{(t)}}=0$, and $x^{(t)},\tilde{s}^{(t-1)}$ are primal-dual feasible solutions to $\LP(\tilde c^{(k)})$, they are also optimal.
Now, it is not hard to see that $\tilde{c}^{(k)} \in \im(A^\top) + c-r$ for some $\0 \leq r\le (m+1)\delta \1$ (\Cref{clm:cost-perturb}). 
Hence, $\tilde{s}^{(t-1)}$ is also an optimal solution to the dual of $\LP(c-r)$. 
The last step of the algorithm consists of identifying the set $N$ of coordinates with large dual slack $\tilde{s}^{(t-1)}_i$.
Then, applying Theorem~\ref{thm:fixing} for $c'=c-r$ allows us to conclude that they can be fixed to zero.

In order to prove \Cref{thm:opt-augment}, we need to show that $N\neq \emptyset$.
Moreover, we need to show that there are at most $T$ iterations of {\sc Ratio-Circuit} per phase.
First, we show that the objective value is monotone nonincreasing.

\begin{lemma}\label{clm:obj-monotone}
For any two iterations $r\geq t$ in phases $\ell\geq k\geq 1$ respectively, 
\[\pr{\tilde{c}^{(\ell)}}{x^{(r)}}\leq \pr{\tilde{c}^{(k)}}{x^{(t)}}.\]
\end{lemma}

\begin{proof}
We proceed by induction on $\ell-k \geq 0$.
For the base case $\ell-k=0$, iterations $r$ and $t$ occur in the same phase.
So, the objective value is nonincreasing from the definition of {\sc Support Circuit} and {\sc Ratio-Circuit}.
Next, suppose that the statement holds for $\ell-k = d$, and consider the inductive step $\ell-k = d+1$.
Let $q$ be the first iteration in phase $k+1$; note that $r\geq q>t$.
Then, we have
$$\pr{\tilde{c}^{(\ell)}}{x^{(r)}}\leq \pr{\tilde{c}^{(k+1)}}{x^{(q)}} \leq \pr{\tilde{s}^{(q-1)}}{x^{(q)}}\leq \pr{\tilde{c}^{(k)}}{x^{(q)}}\leq \pr{\tilde{c}^{(k)}}{x^{(t)}}\, .$$
The first inequality uses the inductive hypothesis. In the second inequality, we use that $\tilde{c}^{(k+1)}$ is obtained from $\tilde{s}^{(q-1)}$ by setting some nonnegative coordinates to 0. The third inequality is by the definition of $\tilde{s}^{(q-1)}$. The final inequality is by monotonicity within the same phase.
\end{proof}

The following claim gives a sufficient condition for \Cref{alg:circ_aug_opt} to terminate.

\begin{claim}\label{clm:terminate}
Let $t$ be an iteration in phase $k\geq 1$. If {\sc Ratio-Circuit} returns an elementary vector $g^{(t)}$ such that $\pr{\tilde{c}^{(k)}}{g^{(t)}} = 0$, then \Cref{alg:circ_aug_opt} terminates in iteration $t+1$.
\end{claim}

\begin{proof}
   Recall that the weights $w$ in \textsc{Ratio-Circuit} are chosen as $w = 1/x^{(t)}$. Recall also the constraint $s^{(t)} \le \lambda w$ in the dual program \eqref{sys:minratio-dual}. Hence, for every $i\in \supp(x^{(t)})$, $s_i^{(t)} x_i^{(t)} \le \lambda = -\pr{\tilde{c}^{(k)}}{g^{(t)}}$, where the equality is due to strong duality.  
   It follows that $\pr{s^{(t)}}{x^{(t)}} \le  -n\pr{\tilde{c}^{(k)}}{g^{(t)}} = 0$.
Since $\tilde{s}^{(t)}, x^{(t+1)}\geq \0$ , we have
\[ 0 \leq \pr{\tilde{s}^{(t)}}{x^{(t+1)}} \leq \pr{s^{(t)}}{x^{(t+1)}} = \pr{s^{(t)}}{x^{(t)}} \leq 0 .\]
Thus, the algorithm terminates in the next iteration.
\end{proof}

The next two claims provide some basic properties of the modified cost $\tilde{c}^{(k)}$.
For convenience, we define $\tilde{c}^{(0)} \coloneqq c$.

\begin{claim}\label{clm:cost-perturb}
For every phase $k\geq 0$, we have $\tilde{c}^{(k)}\in \im(A^{\top})+c-r$ for some $\0\leq r\leq k\delta\1$.
\end{claim}

\begin{proof}
We proceed by induction on $k$.
The base case $k=0$ is trivial.
Next, suppose that the statement holds for $k$, and consider the inductive step $k+1$.
Let $t$ be the first iteration of phase $k+1$, i.e., $\tilde{c}^{(k+1)}_i = \tilde{s}^{(t-1)}_i$ if $i\in S_{t}$, and $\tilde{c}^{(k+1)}_i = 0$ otherwise.
Note that $\tilde{s}^{(t-1)}\in \{\tilde{c}^{(k)},s^{(t-1)}\}$.
Since both of them are feasible to the dual of $\LP(\tilde{c}^{(k)})$, we have $\tilde{s}^{(t-1)} \in \im(A^{\top}) + \tilde{c}^{(k)}$.
By the inductive hypothesis, $\tilde{c}^{(k)} \in \im(A^{\top}) + c-r$ for some $\0\leq r\leq k\delta\1$.
Hence, from the definition of $\tilde{c}^{(k+1)}$, we have $\tilde{c}^{(k+1)}\in \im(A^{\top})+c-r-q$ for some $\0\leq q\leq \delta\1$ as required.
\end{proof}

\begin{claim}\label{clm:cost-norm}
For every phase $k\geq 0$, we have $\|\tilde{c}^{(k)}\|_\infty\leq 3\sqrt{n}\kappa_A$.
\end{claim}

\begin{proof}
We proceed by induction on $k$.
The base case $k=0$ is easy because $\|c\|_\infty \leq \|c\|_2 = 1$.
Next, suppose that the statement holds for $k$, and consider the inductive step $k+1$.
Let $t$ be the first iteration of phase $k+1$.
If $\tilde{s}^{(t-1)} = \tilde{c}^{(k)}$, then $\tilde{c}^{(k+1)}$ is obtained from $\tilde{c}^{(k)}$ by setting some coordinates to 0,
so we are done by the inductive hypothesis.
Otherwise, $\tilde{s}^{(t-1)} = s^{(t-1)}$.
We know that $s^{(t-1)}$ is an optimal solution to \eqref{sys:minratio-dual} for {\sc Ratio-Circuit}$(A, \tilde{c}^{(k)}, 1/x^{(t-1)})$.
Since $c-r \in \im(A^{\top}) + \tilde{c}^{(k)}$ for some $\0 \leq r\leq k\delta\1$ by Claim \ref{clm:cost-perturb}, $s^{(t-1)}$ is also an optimal solution to \eqref{sys:minratio-dual} for {\sc Ratio-Circuit}$(A, c-r, 1/x^{(t-1)})$.
By \eqref{eq:minratio_dual_prox}, we obtain
\begin{align*}
  \|s^{(t-1)}\|_\infty \leq 2\kappa_A \|c-r\|_1 &\le 2\kappa_A(\|c\|_1 + \|r\|_1) \\
  &\le 2\kappa_A\big(\sqrt{n} + nk\delta\big) \le 2\kappa_A\big(\sqrt{n} + n(m+1)\delta\big) \le  3\sqrt{n}\kappa_A.
\end{align*}
The third inequality is due to $\|c\|_2 = 1$, the fourth inequality follows from the fact that there are at most $m+1$ phases, and the last inequality follows from the definition of $\delta$.
\end{proof}

We next show a primal proximity lemma that holds for iterates throughout the algorithm.

\begin{lemma}\label{lem:movement}
Let $t$ be the first iteration of a phase $k\geq 1$. For any iteration $r\geq t$,
\begin{equation}
  \|x^{(r+1)} - x^{(r)}\|_\infty \le \frac{3\sqrt{n}\kappa^2_A}{\delta} \|x^{(t)}_{S_t}\|_1\, .
\end{equation}
\end{lemma}
\begin{proof}
  Fix an iteration $r\geq t$ and let $\ell\geq k$ be the phase in which iteration $r$ occurred.
  Consider the elementary vector $g^{(r)}$.
  If it is returned by {\sc Support-Circuit}, then $g^{(r)}_i<0$ for some $i\in \supp(\tilde{c}^{(\ell)})$ by definition.
  If it is returned by {\sc Ratio-Circuit}, we also have $g^{(r)}_i<0$ for some $i\in \supp(\tilde{c}^{(\ell)})$ unless $\langle \tilde{c}^{(\ell)}, g^{(r)} \rangle = 0$.
  Note that if $\langle \tilde{c}^{(\ell)}, g^{(r)} \rangle = 0$, then the algorithm sets $x^{(r+1)} = x^{(r)}$, which makes the lemma trivially true.
  Hence, we may assume that such an iteration does not occur. 
  
  By construction, we have $x^{(r+1)} - x^{(r)}=\alpha g^{(r)}$ for some $\alpha>0$, and $\alpha |g_i^{(r)}|\le x^{(r)}_i$. Applying the definition of $\kappa_A$ yields
  $$\|x^{(r+1)} - x^{(r)}\|_\infty \leq \kappa_A x^{(r)}_i \leq \frac{\kappa_A}{\delta} \pr{\tilde{c}^{(\ell)}}{x^{(r)}} \leq \frac{\kappa_A}{\delta}\pr{\tilde{c}^{(k)}}{x^{(t)}}  \leq \frac{3\sqrt{n}\kappa^2_A}{\delta} \|x^{(t)}_{S_t}\|_1.$$
  The second inequality uses that all nonzero coordinates of $\tilde{c}^{(\ell)}$ are at least $\delta$.
  The third inequality is by \Cref{clm:obj-monotone}, whereas the fourth inequality is by \Cref{clm:cost-norm} and $\supp(\tilde{c}^{(k)}) = S_t$.
\end{proof}

With the above lemma, we show that any variable which enters $\cL_t$ \emph{at the start of a phase}, is lower bounded by $\poly(n,\kappa_A)\|x^{(t)}_{S_t}\|_1$ in the next $\Theta(mT)$ iterations.

\begin{lemma}\label{clm:stay-positive}
Let $t$ be the first iteration of a phase $k\geq 1$ and let $i\in \mathcal{L}_t\setminus \mathcal{L}_{t-1}$. 
For any iteration $t\leq t'\leq t+2(m+1)T$, %
\[x^{(t')}_i \geq \frac{6\sqrt{n}\kappa^2_A}{\delta}\|x^{(t)}_{S_t}\|_1.\]
\end{lemma}
\begin{proof}
By definition, we have that $x_i^{(t)} \ge \Gamma \|x_{S_t}^{(t)}\|_1$. With \Cref{lem:movement} we get
\begin{equation*}
\begin{aligned}
  x^{(t')}_i \geq x^{(t)}_i - \|x^{(t')} - x^{(t)}\|_\infty &\geq x^{(t)}_i - \sum_{r=t}^{t'-1}\|x^{(r+1)} - x^{(r)}\|_\infty \\
  &\geq \left(\Gamma - \frac{6(m+1)\sqrt{n}\kappa^2_AT}{\delta}\right) \|x^{(t)}_{S_t}\|_1 \geq \frac{6\sqrt{n}\kappa^2_AT}{\delta}\|x^{(t)}_{S_t}\|_1\, .
\end{aligned}
\end{equation*}
The lower bound follows from $T\geq 1$, as long as the constant in the definition of $T$ is chosen large enough.
\end{proof}

For any iteration $t$ in phase $k\ge 1$, let us define
\begin{equation}\label{eq:D-t-def}
D_t:=\bigcup\left\{\cL_{t'}\setminus \cL_{t'-1}\, :\, t' \text{ is the first iteration of phase }k'=1,2,\ldots,k\right\}\, .
\end{equation}
These are the variables which entered $\cL_{t'}$ at the start of a phase for all $t'\leq t$.  
Note that $\rk(D_t)=\rk(\cL_t)$ holds. 
As a consequence of \Cref{clm:stay-positive}, $D_t$ remains disjoint from the support of the modified cost $\tilde{c}^{(k)}$.

\bnote[inline]{BN: Why does what holds for $D_t$ not hold for all the variables in $\mathcal{L}_t$?} 
\bnote[inline]{Mhm, but you could rewrite Lemma 7.6 such that it holds for all guys in $\mathcal{L}_t$. Although, this would probably change a few constants in the proofs. More generally, I feel that there is no need to even define $D_t$, but lets keep it as it would propagate throughout the section. Not worth the effort.}
\bnote[inline]{Ok, not clear how that would be fixed. Can keep it as is.}

\bnote[inline]{Why is the Assumptionin $\langle \tilde{c}^{(k)}, x^{(t)}\rangle>0$ necessary? Shouldn't \Cref{cor:disjoint} hold nonetheless? This assumption also leads to a further case distinction in the main proof of \Cref{thm:opt-augment}.}
\bnote[inline]{I believe it would be cleaner to define $\cL_t\gets \cL_{t-1}\cup \{i\in [n]:\, x^{(t)}_i {\color{red}>} \Gamma \|x_{S_t}^{(t)}\|_1\}$ instead of $\cL_t\gets \cL_{t-1}\cup \{i\in [n]:\, x^{(t)}_i \ge \Gamma \|x_{S_t}^{(t)}\|_1\}$. This would ensure that the variables in $\cL$ are actually positive and would also fix the issue with the assumption in the lemma below, no?}
\bnote[inline]{Ok, let us keep it as is.}

\begin{lemma}\label{cor:disjoint}
Let $0\leq t\leq 2(m+1)T$ be an iteration and let $k\geq 1$ be the phase in which iteration $t$ occured. Let $D_t\subseteq \cL_t$ be defined as in \eqref{eq:D-t-def}.  If $\langle \tilde{c}^{(k)}, x^{(t)} \rangle > 0$, then
\[D_t\cap \supp(\tilde{c}^{(k)}) = \emptyset\, .\]
\end{lemma}

\begin{proof}
For the purpose of contradiction, suppose that there exists an index $i\in D_t\cap \supp(\tilde{c}^{(k)})$. %
Let $r\leq t$ be the iteration in which $i$ was added to $\cL_r$.
By our choice of $D_t$, $r$ is the first iteration of phase $j$ for some $j\leq k$, which implies that $S_r = \supp(\tilde{c}^{(j)})$.
Since $\langle \tilde{c}^{(j)}, x^{(r)} \rangle \geq \langle \tilde{c}^{(k)}, x^{(t)} \rangle >0$ by \Cref{clm:obj-monotone}, we have $\|x^{(r)}_{S_r}\|_1 > 0$.
However, we get the following contradiction
$$ 6\sqrt{n}\kappa^2_A\|x^{(r)}_{S_r}\|_1 \leq \delta x^{(t)}_i\leq \pr{\tilde{c}^{(k)}}{x^{(t)}} \leq \pr{\tilde{c}^{(j)}}{x^{(r)}} \leq 3\sqrt{n}\kappa_A \|x^{(r)}_{S_r}\|_1.$$
The first inequality is by \Cref{clm:stay-positive}, the third inequality is by \Cref{clm:obj-monotone}, while the fourth inequality is by \Cref{clm:cost-norm}.
\end{proof}

The following lemma shows that {\sc Ratio-Circuit} geometrically decreases the norm $\|x^{(t)}_{S_t}\|_1$.

\begin{lemma}\label{clm:norm-decay}
Let $t$ be the first {\sc Ratio-Circuit} iteration in phase $k\geq 1$.
After $\numit \in \mathbb{N}$ consecutive {\sc Ratio-Circuit} iterations in phase $k$, 
\[\|x_{S_{t+\numit}}^{(t+\numit)}\|_1\le \frac{3n^{1.5}\kappa_A}\delta\left(1-\frac1n\right)^{\numit-1} \|x_{\supp(\tilde{c}^{(k)})}^{(t)}\|_1,\]
\end{lemma}
\begin{proof}
  \bnote[inline]{Can't you just pull \Cref{clm:obj-monotone} from line line 2 to line 5? i.e., for the inequality $\langle s^{(t+p-1)}, x^{(t+p)} \rangle \le \langle s^{(t+p-1)}, x^{(t+p-1)} \rangle$. }
\bnote[inline]{Ok.}
\begingroup
\allowdisplaybreaks
\begin{align*}
  \|x^{(t+\numit)}_{S_{t+\numit}}\|_1 &\leq \frac{1}{\delta} \pr{\tilde{s}^{(t+\numit-1)}}{x^{(t+\numit)}} \tag{as $\tilde{s}^{(t+\numit-1)}_i \geq \delta$ for all $i\in S_{t+\numit}$} \\
  &\leq \frac{1}{\delta}\pr{s^{(t+\numit-1)}}{x^{(t+\numit)}} \tag{from the definition of $\tilde{s}^{(t+\numit-1)}$}\\
  &= \frac{1}{\delta}\pr{s^{(t+\numit-1)}}{x^{(t+\numit-1)}+ \alpha g^{(t+\numit-1)}} \tag{for some augmentation step size $\alpha$}\\
  &= \frac{1}{\delta} \left(\pr{s^{(t+\numit-1)}}{x^{(t+\numit-1)}} + \alpha \pr{\tilde{c}^{(k)}}{g^{(t+\numit-1)}} \right) \tag{as $s^{(t+\numit-1)}\in \im(A^\top) + \tilde{c}^{(k)}$} \\
  &\leq \frac{1}{\delta}\pr{s^{(t+\numit-1)}}{x^{(t+\numit-1)}} \tag{because $\pr{\tilde{c}^{(k)}}{g^{(t+\numit-1)}} \leq 0$} \\
  &\leq  -\frac{n}\delta \pr{\tilde{c}^{(k)}}{g^{(t+\numit-1)}} \tag{$s^{(t+p-1)}_i \leq -\pr{\tilde{c}^{(k)}}{g^{(t+p-1)}}/x^{(t+p-1)}_i$ by \eqref{sys:minratio-dual}}\\
  &\leq \frac{n}\delta \left(\pr{\tilde{c}^{(k)}}{x^{(t+\numit-1)}} - \OPT(\tilde{c}^{(k)})\right) \tag{by step size $\alpha \geq 1$ in \Cref{lem:progress}}\\
  &\leq \frac{n}\delta\left(1-\frac1n\right)^{\numit-1}\left(\pr{\tilde{c}^{(k)}}{x^{(t)}}- \OPT(\tilde{c}^{(k)})\right) \tag{by geometric decay in \Cref{lem:progress}}\\
  &\leq \frac{n}\delta\left(1-\frac1n\right)^{\numit-1} \pr{\tilde{c}^{(k)}}{x^{(t)}} \tag{because $\tilde{c}^{(k)}\geq \0$}\\
  &\leq \frac{3n^{1.5}\kappa_A}\delta\left(1-\frac1n\right)^{\numit-1} \|x^{(t)}_{\supp(\tilde{c}^{(k)})}\|_1 \tag{by \Cref{clm:cost-norm}}.
  \end{align*}
\endgroup
\end{proof}

Recall \Cref{lem:big-slack} which guarantees the existence of a coordinate with large dual slack.
It explains why we chose to work with a projected and normalized cost vector in Algorithm \ref{alg:circ_aug_opt}.
We are now ready to prove the main result of this section.

\begin{proof}[Proof of Theorem~\ref{thm:opt-augment}]
We first prove the correctness of \Cref{alg:circ_aug_opt}.
Suppose that the algorithm terminates in iteration $t$.
We may assume that there is at least 1 phase, as otherwise $x^{(0)}$ is an optimal solution to \eqref{sys:lp}.
Let $k\geq 1$ be the phase in which iteration $t$ occurred.
Since $\pr{\tilde{s}^{(t-1)}}{x^{(t)}} = 0$ and $x^{(t)},\tilde{s}^{(t-1)}$ are primal-dual feasible solutions to $\LP(\tilde{c}^{(k)})$, they are also optimal.
By Claim \ref{clm:cost-perturb}, we know that $\tilde{c}^{(k)}\in \im(A^{\top}) + c -r$ for some $\|r\|_\infty\leq (m+1)\delta$.
Hence, $\tilde{s}^{(t-1)}$ is also an optimal dual solution to $\LP(c')$ where $c'\coloneqq c-r$.
Since $c\in \ker(A)$, $\|c\|_2 = 1$, and 
\[\|c-c'\|_\infty\leq (m+1)\delta = \frac{m+1}{2n^{3/2}(m+2)\kappa_A} < \frac{1}{\sqrt{n}(m+2)\kappa_A},\]
where the strict inequality is due to $n\geq m$ and $n>1$, Lemma \ref{lem:big-slack} guarantees the existence of an index $j\in [n]$ such that 
\[\tilde{s}^{(t)}_j > \frac{(m+1)}{\sqrt{n}(m+2)} > (m+1)\kappa_A\|c-c'\|_\infty.\]  %
Thus, the algorithm returns $N\neq \emptyset$.
Moreover, for all $j\in N$, Theorem \ref{thm:fixing} allows us to conclude that $x_j^{(t)} = x^*_j = 0$ for every optimal solution $x^*$ to $\LP(c)$.

Next, we show that if $\rk(\cL_t) = m$ in some phase $k$, then the algorithm will terminate in iteration $r \leq t+ n + 1$.
As long as $\langle \tilde{c}^{(k)}, x^{(t)} \rangle > 0$, we have $D_t\subseteq [n]\setminus \supp(\tilde{c}^{(k)})$ by \Cref{cor:disjoint}.
Moreover, any $i\in \supp(\tilde{c}^{(k)})\cap \supp(x^{(t)})$ induces a circuit in $D_t\cup \{i\}$, so \textsc{Support-Circuit} will be invoked.
Since every call to \textsc{Support-Circuit} reduces $\supp(x^{(t)})$, all the coordinates in $\supp(\tilde{c}^{(k)})$ will be zeroed-out in at most $n$ calls. 
Let $t\leq t'\leq t+n$ be the first iteration when $\langle \tilde{c}^{(k)}, x^{(t')} \rangle = 0$.
Since {\sc Ratio-Circuit} returns $g^{(t')}$ with $\langle \tilde{c}^{(k)}, g^{(t')} \rangle = 0$, the algorithm terminates in the next iteration by \Cref{clm:terminate}.

It is left to bound the number of iterations of \Cref{alg:circ_aug_opt}.
Clearly, there are at most $m+1$ phases.
In every phase, there are at most $n$ \textsc{Support-Circuit} iterations because each call sets a primal variable to 0.
It is left to show that there are at most $T$ \textsc{Ratio-Circuit} iterations in every phase.

Fix a phase $k\geq 1$ and assume that every phase $\ell < k$ consists of at most $T$ many \textsc{Ratio-Circuit} iterations.
Let $t$ be the first iteration in phase $k$.
We may assume that $\rk(\cL_t) < m$, as otherwise there is only one {\sc Ratio-Circuit} iteration in this phase by the previous argument.
Note that this implies $\|x^{(t')}_{S_{t'}}\|_1>0$ for all $t'\leq t$.
Otherwise, $\cL_{t'} = [n]$ and $\rk(\cL_{t'}) = m$, which contradicts $\rk(\cL_{t'}) \leq \rk(\cL_t)$.

Let $r\geq t$ be the first {\sc Ratio-Circuit} iteration in phase $k$. 
Let $D_r\subseteq \cL_r$ be as defined in \eqref{eq:D-t-def}.
By \Cref{clm:stay-positive} and our assumption, we have $x^{(r)}_{D_r} > \0$.
We claim that $D_r\cap \supp(\tilde{c}^{(k)}) = \emptyset$.
This is clearly the case if $\langle \tilde{c}^{(k)}, x^{(r)} \rangle = 0$.
Otherwise, it is given by \Cref{cor:disjoint}.
We also know that there is no circuit in $\supp(x^{(r)})$ which intersects $\supp(\tilde{c}^{(k)})$.
Hence, applying Lemma \ref{lem:circuit-T-indep} with $L=D_r$, $S=\supp(\tilde c^{(k)})$, $x=x^{(r)}$, $z=x^{(r+T)}$ yields
\[
\|x^{(r+T)}_{[n]\setminus \cl(D_r)}\|_\infty \geq \frac{\|x^{(r+T)}_{[n]\setminus \cl(D_r)}\|_1}{n} \ge \frac{\|x^{(r)}_{\supp(\tilde{c}^{(k)})}\|_\infty}{n\kappa_A} \geq \frac{\|x^{(r)}_{\supp(\tilde{c}^{(k)})}\|_1}{n^2\kappa_A}\geq \Gamma{\|x_{S_{r+T}}^{(r+T)}\|_1}\, ,
\]
where the last inequality follows from \Cref{clm:norm-decay} by choosing a sufficiently large constant in the definition of $T$.
Note that $\cl(D_r) = \cl(\cL_r)$ because $D_r$ is a spanning subset of $\cL_r$.
Thus, there exists an index $i\in [n]\setminus \cl(\cL_r)$ which is added to $\cL_{r+T}$, showing that $\rk(\cL_{r+T}) > \rk(\cL_r)$ as required.

Since the main circuit-augmentation algorithm consists of applying Algorithm \ref{alg:circ_aug_opt} at most $n$ times, we obtain the desired bound on the number of iterations.
\end{proof}

%% file: other-circuit.tex
\section{Circuits in General Form}\label{sec:gen-circuits}

There are many instances in the literature where circuits are considered outside standard equality form.
For example, \cite{Borgwardt2016-hierarchy,Kafer2019,DKS19} defined circuits for polyhedra in the general form
\begin{equation}\label{eq:Jesus-P}
P=\{x\in\R^n: Ax=b,\, Bx\le d\}\, ,\end{equation}
 where $A\in \R^{m_A\times n}$, $B\in \R^{m_B\times n}$, $b\in \R^{m_A}$, $c\in \R^{m_B}$. 
It implicitly includes polyhedra in inequality form, which were considered by e.g., \cite{Borgwardt2015,Borgwardt2018circuit}.
 For this setup, they define $g\in\R^n$ to be an elementary vector if 
\begin{enumerate}[(i)]
\item $g\in\ker(A)$, and
\item $Bg$ is support minimal in the collection $\{By: y\in \ker(A), y\neq 0\}$.
\end{enumerate}
In the aforementioned works, the authors use the term `circuit' also for elementary vectors.

Let us assume that 
\begin{equation}\label{assumption:fullrank}
\rk\begin{pmatrix}A\\B\end{pmatrix}=n\, .
\end{equation} 
This assumption is needed to ensure that $P$ is pointed; otherwise, there exists a vector $z\in\R^n$, $z\neq 0$ such that $Az=0$, $Bz=0$. Thus, the lineality space of $P$ is nontrivial. Note that the circuit diameter is defined as the maximum length of a circuit walk between two vertices; this implicitly assumes that vertices exists and therefore the lineality space is trivial.

Under this assumption, we show that circuits in the above definition are a special case of our definition in the Introduction, and explain how our results in the standard form are applicable. 
Consider the matrix and vector
\[
M:=\begin{pmatrix}A&&0\\B&& I_{m_B}\end{pmatrix}\, , \, q:=\begin{pmatrix}b\\ d\end{pmatrix}\, ,
\]
and let $\bar W:=\ker(M)\subseteq\R^{n+m_B}$. Let $J$ denote the set of the last $m_B$ indices, and $W:=\pi_J(\bar W)$ denote the coordinate projection to $J$. The assumption \eqref{assumption:fullrank} guarantees that for each $s\in W$, there is a unique $(x,s)\in \bar W$; further, $x\neq 0$ if and only if $s\neq 0$.

Consider the polyhedron
\[
\bar P=\left\{(x,s)\in \R^{n}\times\R^{m_B}:\, M(x,s)=q\, , s\ge 0  \right\}. 
\]
Note that $P$ is the projection of $\bar P$ onto the $x$ variables. 
Let $Q:=\pi_J(\bar P)\subseteq \R^{m_B}$ be the projection of $\bar P$ onto the $s$ variables.
It is easy to verify the following statements.

\begin{lemma}\label{lem:P-Q-1-2-1}
If \eqref{assumption:fullrank} holds, then 
there is an invertible affine one-to-one mapping $\psi$ between $Q$ and $P$, defined by 
\[
M(\psi(s),s)=q\, .
\]
 Further, $g\in\R^n$ is an elementary vector as in \emph{(i),(ii)} above if and only if there exists $h\in\R^{m_B}$ such that $(g,h)\in \bar W$, $h\neq 0$ and $h$ is support minimal.

Given such a pair $(g,h)\in \bar W$ of elementary vectors,
let $s\in Q$ and let $s':=\aug_Q(s,h)$ denote the result of the circuit augmentation starting from $s$.
Then, $\psi(s')=\aug_P(\psi(s),g)$.
\end{lemma}

Consequently, the elementary vectors of \eqref{eq:Jesus-P} are in one-to-one mapping to elementary vectors in the subspace $W$ as used in this paper.
This was also independently shown by Borgwardt and Brugger \cite[Lemma 3]{BB22}.
By the last part of the statement,  analyzing circuit walks on $P$ reduces to analyzing circuit walks of $Q$ that is given in the subspace form $Q=\{s\in\R^{m_B}:\, s\in W+r, s\ge 0\}$.

Finally, we can represent $Q$ in standard equality form as follows. Using row operations on $M$, we can create an $n\times n$ identity matrix in the first $n$ columns. Thus, we can construct a representation
$Q=\{s\in\R^{m_B}:\, Hs=f, s\ge 0\}$, where $H\in \R^{(m_A+m_B-n)\times m_B}$, $f\in \R^{m_A+m_B-n}$.
By \Cref{lem:P-Q-1-2-1}, 
\[\kappa_H = \max\left\{\frac{|(Bg)_i|}{|(Bg)_j|}: i,j\in \supp(Bg), g \text{ is an elementary vector of \eqref{eq:Jesus-P}}\right\}.\]